\documentclass[10pt]{article}
\usepackage{amssymb}
\usepackage{bbm}
\usepackage{bbding}
\usepackage{stmaryrd}
\usepackage{amscd}
\usepackage{mathrsfs}
\usepackage{amsfonts}
\usepackage{graphicx}
\usepackage{bookmark}
\usepackage{amsmath,amscd,stmaryrd,amssymb,array, amsfonts,amsmath,amssymb}
\usepackage{enumitem}
\usepackage{color}
\usepackage{CJK}
\usepackage{amsthm}
\usepackage{dsfont}

\setenumerate[1]{itemsep=0pt,partopsep=0pt,parsep=\parskip,topsep=5pt}
\setitemize[1]{itemsep=0pt,partopsep=0pt,parsep=\parskip,topsep=0pt}
\setdescription{itemsep=0pt,partopsep=0pt,parsep=\parskip,topsep=5pt}

\numberwithin{figure}{section}
 \numberwithin{equation}{section}

\newtheorem{theorem}{Theorem}[section]
\newtheorem{proposition}[theorem]{Proposition}
\newtheorem{definition}[theorem]{Definition}
\newtheorem{corollary}[theorem]{Corollary}
\newtheorem{lemma}[theorem]{Lemma}
\newtheorem{remark}[theorem]{Remark}
\newtheorem{example}[theorem]{Example}

\newcommand{\cA}{{\mathcal A}}

\newcommand{\cO}{{\mathcal O}}

\newcommand{\cL}{{\mathcal L}}

\newcommand{\0}{{\theta}}

\def\be{\begin{equation}}
\def\ee{\end{equation}}
\def\bes{\begin{equation*}}
\def\ees{\end{equation*}}
\def\bsp{\begin{split}}
\def\esp{\end{split}}

\def\ba{\begin{array}}
\def\ea{\end{array}}
\def\benu{\begin{enumerate}}
\def\eenu{\end{enumerate}}
\def\bt{\begin{theorem}}
\def\et{\end{theorem}}
\def\bp{\begin{proposition}}
\def\ep{\end{proposition}}
\def\bl{\begin{lemma}}
\def\el{\end{lemma}}
\def\br{\begin{remark}}
\def\er{\end{remark}}
\def\bd{\begin{definition}}
\def\ed{\end{definition}}
\def\bc{\begin{corollary}}
\def\ec{\end{corollary}}

\def\va{\varphi}

\def\de{\delta}

\def\lam{\lambda}

\def\O{\Theta}

\def\ve{\varepsilon}

\def\sig{\sigma}

\def\gam{\gamma}

\def\0{\theta}
\def\w{\omega}
\def\a{\alpha}
\def\W{\Omega}

\def\mb{\mbox}

\def\.{\cdot}

\def\R{\mathbb{R}}

\def\A{\forall}

\def\ra{\rightarrow}

\def\~{\tilde}
\def\8{\infty}

\def\E{\mathbb{E}}
\def\mb{\mbox}

\def\<{\langle}
\def\>{\rangle}

\def\Hs{\hspace{1cm}}\def\hs{\hspace{0.5cm}}
\def\Vs{\vskip8pt}\def\vs{\vskip4pt}
\def\[{\left[}\def\]{\right]}
\def\({\left(}\def\){\right)}

\begin{document} 
\begin{center}
{\bf\Large {Critical regularity and dissipativity for stochastic reaction-diffusion equations in Bochner spaces over spaces of continuous functions}}
\end{center}
\begin{center}
 {{Xuewei Ju,\hs Xiaoting Tong}}
\end{center}
\begin{center}
{\footnotesize { Department of Mathematics, Civil Aviation University of China\\
     Tianjin,  China  }}\\
{\footnotesize{\em E-mail}: xwju@cauc.edu.cn.}
\end{center}
\begin{abstract}
In this paper, we consider the stochastic reaction-diffusion equation $\mathrm{d}u = -A u + f(u))\mathrm{d}t + \sigma(u)\mathrm{d}W$ on a smooth bounded domain $\mathcal{O}$ with homogeneous Dirichlet boundary conditions. We investigate the long-time behavior of solutions with a strongly dissipative drift nonlinearity and superlinear multiplicative noise in the Bochner space $L^q(\Omega; C_0(\overline{\mathcal{O}}))$, $q \ge 2$. Here  $A=-\Delta$ be the Laplace operator on $\mathcal O$ with homogeneous Dirichlet boundary conditions and $W$ is a two-sided trace-class Wiener process. The standard Galerkin method fails to yield energy estimates in $L^q(\Omega; L^q(\mathcal{O}))$ via the It\^o formula for $q > 2$, owing to the interference of projection operators when dealing with nonlinear terms; meanwhile, the classical theory of mild solutions lacks sufficient spatial regularity to apply the It\^o formula directly. To overcome these difficulties, we consider mild solutions and establish a critical regularity estimate for the corresponding stopped process $u_n(t)$ in $W_0^{1,q}(\mathcal{O})$, which rigorously justifies the use of the It\^o formula in the non-Hilbert space $L^q(\Omega; L^q(\mathcal{O}))$. As a result, we derive explicit moment energy estimates and quantitative dissipativity bounds, yielding global existence, uniqueness, and exponential asymptotic decay of solutions in $L^q(\Omega; C_0(\overline{\mathcal{O}}))$. Unlike previous qualitative results in continuous function spaces, our framework provides a fully quantitative theory of global dissipativity.
\end{abstract}

\section{Introduction}
In physical models such as chemical reactions and population dynamics, the {\sl pointwise behavior} of solutions (e.g., maximum concentration, spatial distribution) carries clear physical significance. Such models are often described by reaction-diffusion equations, where the unknown function represents a concentration or density that is naturally pointwise defined. This motivates us to work in spaces of continuous functions, where solutions are pointwise defined and boundary conditions can be naturally incorporated. These spaces provide a more direct framework for problems where the values of the solution at each point matter.

Specifically, we study the following stochastic reaction-diffusion equation in the space of continuous functions:
\begin{equation}\label{eq1}
    \mathrm{d}u = \left(\cA u + f(u)\right)\mathrm{d}t + \sigma(u)\mathrm{d}W, \quad u|_{\partial\mathcal{O}} = 0,
\end{equation}
where $\mathcal{O} \subset \mathbb{R}^d$ is a smooth bounded domain,  $\cA$ is a second-order self-adjoint elliptic operator, $f,\sigma\in C^1(\mathbb{R})$ are superlinear functions, and $W$ is a two-sided trace-class Wiener process defined on a complete filtered probability space satisfying condition \hyperref[H1]{(H1)} (see Section~\ref{se2} for details). The assumptions on $f$ and $\sigma$ will be formalized as \hyperref[H2]{(H2)} and \hyperref[H3]{(H3)} in Section~\ref{se4}.

For stochastic equations driven by nonlinear multiplicative noise, the It\^o formula plays an essential role in establishing global existence and dissipativity. A standard approach is to apply the It\^o formula to Galerkin approximations, derive energy estimates, and then pass to the limit $n\to\infty$ to obtain weak solutions in the sense of distributions. This approach has been widely adopted in the setting of Bochner spaces over Hilbert spaces,  particularly after  Kloeden \& Lorenz \cite{KL1} and Wang   \cite{WBX} developed the theory of mean random dynamical systems for stochastic equations driven by nonlinear multiplicative noise, leading to a number of results under Lipschitz or monotonicity conditions \cite{CW,Gu2,KL2, QSW,WCT,WBX',WZ,WK}. One advantage of this method is that Galerkin approximations enjoy sufficient spatial regularity to justify the use of the It\^o formula.

In this paper, we investigate problem \eqref{eq1} in the space $L^q(\Omega; C_0(\overline{\mathcal{O}}))$ with $q\ge 2$, and estimates in $L^q(\Omega; L^q(\mathcal{O}))$ play a central role in our analysis. However, for $q>2$, the space $L^q(\mathcal{O})$ is not a Hilbert space, which causes a fundamental difficulty for the Galerkin method when handling nonlinear terms.
The difficulty can be explained as follows. Let $P_N$ denote the orthogonal projection from $L^2(\mathcal{O})$ onto $H_N$, the subspace spanned by the first $N$ eigenfunctions of the Laplacian. In the Galerkin approximation, the nonlinear terms become $P_N f(u_N)$ and $P_N \sigma(u_N)$, where $u_N = P_N u$ is the finite-dimensional approximation. To derive energy estimates, one usually applies the It\^o formula to the functional $\|u_N\|_{L^q}^q = \int_{\mathcal{O}} |u_N|^q dx$. This yields a term of the form
$
\int_{\mathcal{O}} |u_N|^{q-2} u_N \cdot P_N f(u_N) dx
$.
For $q=2$, the projection $P_N$ can be removed by self-adjointness and the identity $u_N = P_N u_N$, reducing the expression to $\int_{\mathcal{O}} f(u_N) u_N dx$. For $q>2$, however, neither $f(u_N)$ nor $|u_N|^{q-2} u_N$ is necessarily in $H_N$, so the projection cannot be eliminated. The same obstruction appears for the diffusion term $P_N \sigma(u_N)$. Consequently, for $q>2$, the presence of $P_N$ prevents the application of the dissipativity condition \eqref{eq3.9} (cf.\ Hypothesis \hyperref[H2]{(H2)} in Section~\ref{se4}), thus blocking the derivation of dissipativity estimates via the Galerkin method.

To this end, we adopt the framework of {\sl mild solutions}, which are constructed directly via the semigroup $S(t)$ and stochastic convolutions--thus completely circumventing the issue of projection operators $P_N$. Nevertheless, mild solutions introduce a new challenge: their spatial regularity is insufficient to directly apply the It\^{o} formula.
Specifically, applying the It\^{o} formula requires $u_n(t) \in W_0^{1,q}(\mathcal{O}) = D(A_q^{1/2})$, where $u_n$ denotes the stopped process associated with the mild solution $u$ via the stopping time $\tau_n := \inf\{t > 0 : \|u(t)\|_{C_0} \geq n\}$, and $A_q$ is the realization of $-\mathcal{A}$ in $L^q(\mathcal{O})$ with Dirichlet boundary conditions. However, while estimates for $\|A_q^\alpha u_n\|_{L^q}$ with $0 < \alpha < 1/2$ can be derived relatively easily, the critical case $\alpha = 1/2$ is more delicate.

In fact, the criticality of the exponent $\alpha = 1/2$ originates from the analysis of stochastic integrals. In the BDG framework, estimating $\mathbb{E}\|A_q^\alpha \int_0^t S(t-s)\Phi(s) \, dW(s)\|_{L^q}^q$ reduces to an integral of the form $\int_0^t (t-s)^{-2\alpha} e^{-\lambda (t-s)} \|\Phi(s)\|_{L^q}^2 ds$ for some $\lambda>0$, where $\|\Phi(s)\|_{C_b} \le C$. 
When $0 < \alpha < 1/2$, we have $0 < 2\alpha < 1$, so the kernel $(t-s)^{-2\alpha}$ is integrable near $s=t$, and the integral converges without extra regularity of $\|\Phi(s)\|_{C_b}$. When $\alpha = 1/2$, the kernel becomes $(t-s)^{-1}$, which is non-integrable near $s=t$; convergence then depends entirely on the decay of $\|\Phi(s)\|_{C_b}$ near $s=t$. 
In other words, the regularity of $\|\Phi(s)\|_{C_b}$ must compensate for the kernel singularity. This observation identifies $\alpha = 1/2$ as a critical exponent in the analysis of stochastic integrals for SPDEs.

To the best of our knowledge, global explicit moment estimates at $\alpha = 1/2$ for standard semilinear SPDEs have not been established in the literature. In this paper, assuming initial data in $D(A_q^{1/2})$, we first prove the following global critical regularity estimate for the stopped process $u_n$:
\begin{equation}\label{eq01}
\mathbb{E}\int_0^T\|A_q^{1/2}u_n(t)\|_{L^q}^qdt <\8, \Hs \forall T>0.
\end{equation} Then, using an approximation argument detailed in Section~\ref{se4}, we relax this initial regularity condition to the natural space $C_0(\overline{\mathcal{O}})$.

The proof of estimate \eqref{eq01} relies on a refined decomposition of $\sigma(u_n(s))$ designed to cancel the singularity of the stochastic integral kernel at $\alpha = 1/2$:
$$
\sigma(u_n(s)) = \bigl[\sigma(u_n(s))-\sigma(0)\bigr] + \sigma(0).
$$
Accordingly, the stochastic convolution splits into two parts. The first convolution is estimated using the H\"older continuity of $\sigma(u_n(s))$ in suitable spaces, while the estimate of the second stochastic convolution, corresponding to the constant part $\sigma(0)\neq 0$, relies on a weaker noise intensity condition (i.e., condition \hyperref[H1*]{(H1*)} in Section \ref{se3} is needed). With the critical regularity estimate \eqref{eq01} at hand, we are able to apply the It\^o formula to obtain dissipative moment estimates for $u_n$.  Letting $n \to \infty$ and using the approximation procedure, we can extend these estimates to the original mild solution $u$.

However, obtaining the above critical regularity estimate comes at a cost: as noted above, the initial data must belong to a proper subspace of $C_0(\overline{\mathcal{O}})$; moreover, when $\sigma(0) \neq 0$, the intensity of the Brownian motion must be reduced. These additional conditions are not our final goal.

To return to the desired original conditions, we resort to a {\sl double-index approximation} argument. Since $L^q(\Omega; C_0(\overline{\mathcal{O}}))$ is non-reflexive,  one cannot obtain solution estimates directly from energy estimates and weak convergence. To resolve this, we construct a strongly convergent approximating sequence: we first approximate $u_0 \in  C_0(\overline{\mathcal{O}})$ by a sequence $u_0^m \in D(A_0^{1/2})$; and simultaneously approximate $\mu_j$ by $\mu_{jm}$ satisfying \hyperref[H1*]{(H1*)} such that
$$
\Theta_{mn} := \sum_{j=1}^\infty |\mu_{jm} - \mu_{jn}| \,\|e_j\|_{C_0}^2 \to 0 \Hs \text{as } m,n\to\infty.
$$
Using the uniform energy estimates, we show that $\{u^m(t)\}$ converges strongly to a limit $u(t)$, which is precisely the unique mild solution. To obtain the dissipativity estimates for $u(t)$ in $L^q(\Omega; C_0(\overline{\mathcal{O}}))$, a careful computation is required to obtain uniform estimates for $\{u^m(t)\}$ that are independent of $m$.
\vs
To illustrate our main result more clearly, we consider the following special case.
\begin{theorem}
Let
$$
f(u) = \sum_{j=1}^{\beta} b_j u|u|^{j-1}, \Hs \text{with } b_\beta < 0,
$$
and let $\sigma$ be a polynomial satisfying
$$
|\sigma(u)| \le c(|u|^\chi + 1).
$$
Assume $\chi \ge 1$ and $\beta + 1 > 2\chi$, and \hyperref[H1]{(H1)}.
Then for any initial data $u_0 \in C_0(\overline{\mathcal{O}})$ and $q\ge2$, the $C_0(\overline{\mathcal{O}})$-valued mild solution of equation \eqref{eq1} exists globally in time, and the following dissipative estimate
$$
\mathbb{E}\big(\| u(t) \|_{C_0}^q\big) \le C \Big( \| u_0 \|_{L^{q\beta}}^{q\beta} e^{-\varrho_{q\beta}(t-1)} + 1 \Big), \Hs t \ge 1,
$$
 where the constant $\varrho_{q\beta}>0$ depends on $q$, $\beta$, and $b_\beta$, while the constant $C>0$ depends on $q$, $d$, $\beta$, $|\mathcal{O}|$, 
and the intensity of $W(t)$.
\end{theorem}

The study of SPDEs in non-Hilbert spaces has been developed under various frameworks.
Cerrai \cite{Cerrai2003} established global well-posedness and qualitative moment bounds for  stochastic reaction-diffusion equations in continuous function spaces, allowing superlinear drift but requiring the noise coefficient to be globally Lipschitz (linearly growing). However, due to limited spatial regularity (only H\"older continuity) and the non-Hilbert structure of the state space, the It\^o formula is not applicable, leading only to qualitative results such as existence, uniqueness, and uniform moment boundedness.

In recent years, Salins \cite{Salins2021,Salins2023,Salins2025,Salins2025'} studied stochastic reaction-diffusion equations driven by superlinear multiplicative noise, and established global existence results in spaces of continuous or bounded functions. His proofs rely on factorization formulas and stopping-time sequences. In \cite{Salins2021}, he considered superlinear multiplicative noise and introduced a strong dissipativity condition balancing the growth rates of the drift and the noise. He proved that sufficiently strong dissipation can counteract the expansive effect of superlinear noise and prevent blowup. However, the quantitative asymptotic decay of the solutions was not addressed.

Agresti \& Veraar \cite{Agresti2022a, Agresti2022b} established local well-posedness for quasilinear and semilinear SPDEs in critical spaces using stochastic maximal $L^p$-regularity theory, allowing rough initial data and polynomial growth nonlinearities. Their results focus on local existence and instantaneous regularization, without addressing quantitative dissipativity of global solutions.

In contrast to these works, we analyze the long-time behavior of \eqref{eq1} in the Bochner space $L^q(\Omega; C_0(\overline{\mathcal{O}}))$ under superlinear multiplicative noise and a strongly dissipative drift nonlinearity.
A core innovation of our approach is the critical regularity estimate at $\alpha = 1/2$, which ensures that the stopped process $u_n(t)$ belongs to $W^{1,q}_0(\cO)$.
This regularity allows the application of the It\^o formula in the $L^q(\cO)$ space setting, a tool unavailable in previous studies.
Using the It\^o formula, we establish novel moment-energy estimates, which upgrade qualitative analyses in the literature to a rigorous quantitative framework.
The resulting dissipativity estimates yield sharp exponential decay rates and provide a foundation for studying the uniqueness of invariant measures and exponential ergodicity.
Compared with Salins \cite{Salins2021}, who considered a broader class of noise including space-time white noise, we focus on trace-class Wiener noise and improve his results by establishing quantitative exponential dissipativity estimates with explicit moment bounds.
 
The remainder of this paper is structured as follows. In Section~\ref{se2}, we prove the local existence of mild solutions in $C_0(\overline{\mathcal{O}})$ via a stopping-time argument. Section~\ref{se3} establishes the critical regularity estimate at $\alpha = 1/2$, which provides the spatial regularity needed to apply the It\^o formula. Based on these estimates, Section~\ref{se4} proves the global existence and mean dissipativity of solutions in $L^q(\Omega; C_0(\overline{\mathcal{O}}))$, yielding quantitative exponential decay bounds. Finally, an approximation argument is used to relax the initial regularity condition to the natural space $L^q(\Omega; C_0(\overline{\mathcal{O}}))$.

\section{Local existence of mild solutions in $C_0(\overline{\cO})$}\label{se2}
Let $\mathcal{O} \subset \mathbb{R}^d$ be a smooth bounded domain, and let $A=-\Delta$ be the Laplace operator on $\mathcal O$ with homogeneous Dirichlet boundary conditions.  Its realization $A_2$ in $L^2(\mathcal{O)}$ is self-adjoint with compact resolvent, and possesses a sequence of eigenvalues $\{\lam_j\}_{j=1}^\infty$ and corresponding eigenfunctions $\{e_j\}_{j=1}^\infty$ that form an orthonormal basis of $L^2(\mathcal{O})$ and satisfy
$$
0 < \lam_1 \leq \lam_2 \leq \cdots \leq \lam_j \to \infty \Hs \text{as} \Hs j \to \infty.
$$
It is well known that $\{e_j\}_{j=1}^\infty \subset C_0(\overline{\mathcal{O}})$, where
$$
C_0(\overline{\mathcal{O}}):=\{u\in C(\overline{\mathcal{O}}): u|_{\partial\mathcal{O}}=0\}.
$$

Let $W(t)$ be a Wiener process on a filtered probability space $(\Omega,\mathcal{F},\{\mathcal{F}_t\}_{t\in\mathbb{R}},\mathbb{P})$, taking values in $L^2(\mathcal{O})$. Specifically,
$$
W(t, x)=\sum_{j=1}^{\infty} \sqrt{\mu_j}\, e_j(x) B_j(t), \Hs t\in\mathbb{R},
$$
where $\mu_j\ge0$, $j=1,2,\cdots$ and $\{B_j(t)\}_{j \in \mathbb{N}}$ are independent one-dimensional Brownian motions on $(\Omega,\mathcal{F},\{\mathcal{F}_t\}_{t\in\mathbb{R}},\mathbb{P})$. 

\vs
\medskip
The following condition on the noise intensity is assumed throughout.
 
\hypertarget{H}{}\paragraph{(H1)} \label{H1}
Suppose that
$$
\Theta := \sum_{j=1}^{\infty} \mu_j \|e_j\|_{C_0}^2 < \infty,
$$
where $\|\cdot\|_{C_0}$ denotes the supremum norm of $C_0(\overline{\mathcal{O}})$.
\vs

We consider \eqref{eq1} in the space $C_0(\overline{\cO})$ and reformulate it in abstract form as
\be\label{eq}d u = (-A_0 u + f(u)) dt+ \sigma(u)  dW
 \ee with initial data $u(0) = u_0\in C_0(\overline{\cO})$, where $A_0$ is the  realization of $A=-\Delta$ in $C_0(\overline{\cO})$, $f,\sig\in C^1(\R)$  are superlinear functions.
 \bd\label{de1}
A $C_0(\overline{\cO})$-valued process $u(t)$ is local mild solution to \eqref{eq} if
\be\label{eq1.1'}
u(t) = S(t)u_0 + \int_0^t S(t-s)f(u(s))\,ds + \int_0^t S(t-s)\sigma(u(s))\,dW(s)
\ee for all $t\in[0,\tau_n]$,
where $\{S(t)\}_{t\ge0}$ is the semigroup on $C_0(\overline{\cO})$ generated by $A_0$ (see \eqref{eq2.1}) for definition),  $\tau_n$ is the stopping time $$\tau_n := \inf\{t > 0 : \|u(t)\|_{C_0} \geq n\}.$$  

The random time $\tau(\w):=\lim_{n\ra\8}\tau_n(\w)$ is called the maximal existence time.
\ed
\bp\label{th1} \cite[Theorem 5.3]{Cerrai2003} Assume \hyperref[H1]{(H1)}. Then for any initial datum $u_0\in C_0(\overline{\mathcal{O}})$, there exists, for almost every $\omega\in\Omega$, a maximal existence time $0<\tau(\omega)\le\infty$ such that \eqref{eq} has a unique local mild solution $u\in C([0,\tau(\omega));C_0(\overline{\mathcal{O}}))$ in the sense of Definition \ref{de1}.\ep

\section{Critical regularity of mild solutions}\label{se3}

\subsection{Realizations of the negative Laplacian and their fractional powers}\label{sub1}
Let $C_b(\mathcal{O})$ denote the Banach space consisting of all continuous and bounded functions on $\mathcal{O}\subset \mathbb{R}^d$, endowed with the supremum norm $\|\cdot\|_{C_b}$. 
Note that $C_0(\overline{\mathcal{O}})$ is a closed subspace of $C_b(\mathcal{O})$, and  the norm $\|\cdot\|_{C_b}$ coincides with $\|\cdot\|_{C_0}$ on $C_0(\overline{\mathcal{O}})$. 
\Vs
 The operator $A=-\Delta$ can be realized in different function spaces with corresponding domains:
\begin{equation*}
\left\{
\begin{array}{ll}
A_q: D(A_q)\subset L^q(\mathcal{O}) \to L^q(\mathcal{O}), & D(A_q)=W^{2,p}(\mathcal{O})\cap W^{1,p}_0(\mathcal{O}),\Hs 1<p<\8;\\[2ex]
A_b: D(A_b)\subset C_b(\mathcal{O}) \to C_b(\mathcal{O}), & D(A_b)=\big\{ u : u \in \bigcap_{q\ge 1} W^{2,q}(\mathcal{O}), \; \cA u \in C_b(\mathcal{O}),\; u|_{\partial \mathcal{O}} = 0 \big\};\\[2ex]
A_0: D(A_0)\subset C_0(\overline{\mathcal{O}}) \to C_0(\overline{\mathcal{O}}), & D(A_0)=\left\{ u\in D(A_b):\cA u\in C(\overline{\mathcal{O}}),\;\cA u|_{\partial \mathcal{O}} = 0 \right\}.
\end{array}
\right.
\end{equation*}

It is straightforward to verify the inclusion relations $D(A_0)\subset D(A_b)\subset D(A_q)$, which follows from the continuous embeddings $C_0(\overline{\mathcal{O}})\hookrightarrow C_b(\mathcal{O})\hookrightarrow L^q(\mathcal{O})$. Moreover, any two realizations of $A$ coincide on their common domain.

For each $*\in \{p,b,0\}$, the realization $A_*$ is a positive sectorial operator in its respective underlying space $X_*$, where $X_q:=L^q(\mathcal{O})$, $X_b:=C_b(\mathcal{O})$, and $X_0:=C_0(\overline{\mathcal{O}})$ (see \cite[Corollary 3.1.21 (ii)]{Lu}). By the theory of sectorial operators, $-A_*$ generates a bounded analytic semigroup $\{S(t)\}_{t\ge 0}$ on $X_*$. This semigroup admits an integral representation via the Dirichlet heat kernel $G:\mathcal{O}\times \mathcal{O}\times (0,\infty)\to \mathbb{R}$, i.e.,
\be\label{eq2.1}
(S(t)f)(x)=\int_{\mathcal{O}} G(x,y,t)f(y)\,dy,\Hs \forall f\in X_*, \ t>0, \ x\in \mathcal{O}.
\ee
Notably, the semigroup $\{S(t)\}_{t\ge 0}$ is independent of the specific realization space $X_*$, in the sense that it acts consistently on functions belonging to the intersections of these spaces.
\begin{remark}\cite{Pazy}
The semigroup $\{S(t)\}_{t\ge0}$ is strongly continuous on $L^q (\mathcal{O})$ ($1<p<\infty$) and on $C_0(\overline{\mathcal{O}})$, but not on $C_b(\mathcal{O})$. The strong continuity on $X_*$ in the case $*\in\{p,0\}$ is equivalent to the denseness condition $\overline{D(A_*)}=X_*$.
\end{remark}
\vs
For any $\alpha>0$, let $A_*^\alpha$ denote the fractional power of $A_*$, which is defined as the inverse of $A_*^{-\alpha}$ (see Definition \ref{d:4.1}). The following fundamental properties hold for each $*\in \{p,b,0\}$:
\begin{itemize}
    \item Semigroup property: $S(t)S(s)=S(t+s)$ for all $t,s\ge 0$;
    \item Exponential stability: There exists a positive constant $C$ such that
    \begin{equation}\label{eq2.13}
    \|S(t)\|_{\mathcal{L}(X_*)}\le Ce^{-\lam t},\Hs t\ge0,
    \end{equation}
    where $0<\lam<\lam_1$ and $\lam_1>0$ is the first eigenvalue of $A_*$;
    \item Smoothing effect: $S(t)u\in D(A_*)$ for each $t>0$ and $u\in X_*$;
    \item Domain of fractional powers: $D(A_*^\alpha)$ is a subspace of $X_*$, equipped with the norm $\|u\|_{D(A_*^\alpha)}=\|A_*^\alpha u\|_{X_*}$ for $u\in D(A_*^\alpha)$; in particular, $D(A_q^{1/2})=W^{1,q}_0(\mathcal{O})$ for $1<p<\infty$;
    \item Monotonicity of domains: $D(A_*^\beta)\subset D(A_*^\alpha)$ whenever $\beta>\alpha>0$;
    \item Commutativity:
    $$
    A_*^\alpha S(t)=S(t)A_*^\alpha \Hs\text{on } D(A_*^\alpha),\ \forall t\ge0;
 $$
    \item Smoothing estimate for fractional powers: There exists a constant $C_\alpha>0$ such that
    \begin{equation}\label{eq2.11}
    \|A_*^\alpha S(t)\|_{\mathcal{L}(X_*)}\le C_\alpha t^{-\alpha}e^{-\lam t},\Hs t>0,
    \end{equation}
    where $0<\lam<\lam_1$ is the same as that in \eqref{eq2.13};
    \item Modulus of continuity estimate: For any $u\in D(A_*^\gamma)$ with $0<\gamma\le 1$, there exists a constant $C_\gamma>0$ such that
    \begin{equation}\label{eq2.12}
    \|(S(t)-I)u\|_{X_*}\le C_\gamma t^{\gamma}\|A_*^\gamma u\|_{X_*},\Hs t\ge0;
    \end{equation}
    \item Additivity of exponents: $A_*^{\alpha+\beta}=A_*^\alpha A_*^\beta=A_*^\beta A_*^\alpha$ on $D(A_*^{\alpha+\beta})$ for any $\beta>0$.
\end{itemize}

These properties are standard for analytic semigroups generated by sectorial operators. Their proofs follow from the representation of $A_*^{-\alpha}$ (see Definition \ref{d:4.1}) and the estimates for $A_*S(t)$ (see Proposition \ref{p2}), and detailed proofs can be found in \cite[Section 1.4]{D.H}.

\subsection{Critical regularity under additional assumptions on noise and initial conditions}
Since the norm $\|\cdot\|_{C_b}$ coincides with $\|\cdot\|_{C_0}$ on the subspace $C_0(\overline{\mathcal{O}})\subset C_b(\cO)$, for consistency of notations we may simply use $\|\cdot\|_{C_b}$ even when referring to elements of $C_0(\overline{\mathcal{O}})$, and make no distinction between the two norms in the sequel. For instance, we may write $\|e_j\|_{C_b}$ instead of $\|e_j\|_{C_0}$ for $j\in\mathbb{N}^+$.
\vs

For a.e. $\w\in\W$, let $u(t)$, $t\in[0,\tau(\w))$ denote the mild solution of \eqref{eq}, where $0<\tau(\w)\le\8$  is  the maximal existence time. For any $n\in \mathbb{N}^+$, denote by $u_n(t):=u(t\wedge\tau_n)$, $t\ge0$, where $\tau_n := \inf\{t > 0 : \|u(t)\|_{C_b} \geq n\}$. Then 
\begin{equation}\label{eq1.1}
\begin{split}u_n(t) =\,&\,S(t\wedge\tau_n)u_0 + \int_0^{t\wedge\tau_n} S(t\wedge\tau_n-s)f(u_n(s))\,ds\\
&+ \int_0^{t\wedge\tau_n} S(t\wedge\tau_n-s)\sigma(u_n(s))\,dW(s)\\
=:&\,S(t\wedge\tau_n)u_0 +I_1(t)+I_2(t).\end{split}\end{equation}
\vs
In this part, we impose a stronger assumption on $W(t)$ than \hyperref[H1]{(H1)} in the sense that $W(t)$ has weaker intensity when $\sigma(0) \neq 0$.
\hypertarget{H}{}\paragraph{(H1*)} \label{H1*} (Weaker noise intensity condition) Either
\begin{itemize}
    \item $\displaystyle \Theta:=\sum_{j=1}^\infty \mu_j\|e_j\|_{C_b}^2<\infty$ \quad\text{and}\quad $\sigma(0)=0$, or
    \item$\displaystyle \Theta':=\sum_{j=1}^\infty \lambda_j^\delta \mu_j\|e_j\|_{C_b}^2<\infty$ for some $0<\delta<1$ \quad\text{and}\quad $\sigma(0)\ne0$.
\end{itemize}

The main result of this section is stated below. It establishes critical regularity of the solutions under the above weaker noise intensity condition and higher regularity of the initial data. This provides an essential prerequisite for deriving energy estimates via It\^o formula later.

\begin{theorem}\label{th2}
Assume \hyperref[H1*]{(H1*)} and $q>1$. Then for any $T>0$ and $u_0\in D(A_0^{1/2})$, there exists a constant $C>0$ such that
$$
\mathbb{E}\big(\|A_q^{1/2}u_n(t)\|_{L^q}^q\big) \le C,\qquad 0\le t\le T,
$$
where the constant $C>0$ depends on $\|A_0^{1/2}u_0\|_{L^q}$, $T$, $q$, $d$, $n$, $|\mathcal O|$, $\Theta$, $\Theta'$, $\delta$ and $\lambda$, but is independent of $\Theta'$ and $\delta$ when $\sigma(0)=0$.
\end{theorem}

Throughout the rest of this section, we always work under Assumption \hyperref[H1*]{(H1*)}. For notational brevity, this assumption will be omitted from the statements of the subsequent lemmas and corollary.
\Vs
In the paper, when dealing with the stochastic integral
$$
I_2(t)=\int_0^{t\wedge\tau_n}S(t\wedge\tau_n-s)\sigma(u_n(s))\,dW(s),\qquad t\ge0,
$$
we often rely on the estimate of the auxiliary adapted process
$$
Z_\nu^n(t):=\int_0^t(t-s)^{-\nu}S(t-s)\sigma(u_n(s))\mathbbm{1}_{\{s\le\tau_n\}}\,dW(s),\qquad t\ge0,
$$
where $\nu\in(0,1)$.
By the standard factorization method (see, e.g., Da Prato--Zabczyk \cite[Sec. 5.3]{DZ}), we have
\be\label{eq3.10}
I_2(t)=\frac{\sin(\pi\nu)}{\pi}\int_0^{t\wedge\tau_n}(t\wedge\tau_n-s)^{\nu-1}S(t\wedge\tau_n-s)Z_\nu^n(s)\,ds.
\ee

\begin{lemma}\label{lem:Znu-estimate}
Suppose that $0<\alpha<1/2$ and $q\ge2$. Then for any $\nu\in(0,1/2-\alpha)$, there exists a constant $C>0$, depending only on $\alpha,q,\nu,|\mathcal O|,\Theta,n,\lambda$, such that
$$
\mathbb{E}\big(\|A_q^\alpha Z_\nu^n(t)\|_{L^q}^q\big) \le C,\qquad t\ge0.
$$
\end{lemma}

\begin{proof}
 For fixed $x\in\mathcal O$, the process $A_q^\alpha Z_\nu^n(t)(x)$ is a real-valued It\^o integral, and can be written as
$$
A_q^\alpha Z_\nu^n(t)(x)
=\sum_{j=1}^\infty\sqrt{\mu_j}\int_0^t(t-s)^{-\nu}\mathbbm{1}_{\{s\le\tau_n\}}
\big[A_q^\alpha S(t-s)\sigma(u_n(s))e_j\big](x)\,dB_j(s),
$$
where the series and stochastic integral commute; this is justified by the condition $\Theta<\infty$. Applying the Burkholder--Davis--Gundy (BDG) inequality, we obtain, for each $x\in\mathcal O$,
$$
\mathbb{E}\big(\big|A_q^\alpha Z_\nu^n(t)(x)\big|^q\big)
\le C_q\,\mathbb{E}\Bigg[\Bigg(\sum_{j=1}^\infty\mu_j
\int_0^t(t-s)^{-2\nu}\mathbbm{1}_{\{s\le\tau_n\}}
\big|[A_q^\alpha S(t-s)\sigma(u_n(s))e_j](x)\big|^2\,ds\Bigg)^{q/2}\Bigg].
$$
Integrating over $x\in\mathcal O$ and applying Minkowski's inequality in the form
$$
\int_{\mathcal O}\left(\int_0^t |F(s,x)|\,ds\right)^{q/2}dx
\le
\left(\int_0^t \left(\int_{\mathcal O} |F(s,x)|^{q/2}dx\right)^{2/q}ds\right)^{q/2},
$$
where
$$
F(s,x):=\sum_{j=1}^\infty\mu_j(t-s)^{-2\nu}\mathbbm{1}_{\{s\le\tau_n\}}
\big|[A_q^\alpha S(t-s)\sigma(u_n(s))e_j](x)\big|^2,
$$
to interchange the $L^q$-norm with the time integral, we arrive at
$$
\begin{aligned}
\mathbb{E}\big(\|A_q^\alpha Z_\nu^n(t)\|_{L^q}^q\big)
&\le C_q\,\mathbb{E}\Bigg[\Bigg(\sum_{j=1}^\infty\mu_j
\int_0^t(t-s)^{-2\nu}\mathbbm{1}_{\{s\le\tau_n\}}
\|A_q^\alpha S(t-s)\sigma(u_n(s))e_j\|_{L^q}^2\,ds\Bigg)^{q/2}\Bigg].
\end{aligned}
$$

Using the estimate $\sum_{j=1}^\infty\mu_j\|e_j\|_{L^q}^2\le |\mathcal O|^{2/q}\Theta$, the smoothing estimate $\|A_q^\alpha S(r)\|_{\mathcal L(L^q)}\le C_\alpha r^{-\alpha}e^{-\lambda r}$ and the bound $\|u_n(s)\|_{C_b}\le n$, we get
\be\label{eq3.6}
\begin{aligned}
\mathbb{E}\big(\|A_q^\alpha Z_\nu^n(t)\|_{L^q}^q\big)
&\le C_{q,|\mathcal O|}\Theta^\frac{q}{2}\,
\E\Bigg[\Bigg(\int_0^t(t-s)^{-2(\nu+\alpha)}e^{-2\lambda(t-s)}\|\sig(u_n(s))\|_{L^q}^2\mathbbm{1}_{\{s\le\tau_n\}}\,ds\Bigg)^{q/2}\Bigg] \\
&\le C_{q,|\mathcal O|,n}\Theta^\frac{q}{2}\,
\Bigg(\int_0^t(t-s)^{-2(\nu+\alpha)}e^{-2\lambda(t-s)}\,ds\Bigg)^{q/2} \\&\le C_{q,|\mathcal O|,n,\Theta,\nu,\alpha,\lambda},
\end{aligned}
\ee
where the last inequality follows from the uniform boundedness of the integral
$$
\int_0^t(t-s)^{-2(\nu+\alpha)}e^{-2\lambda(t-s)}\,ds
\le \int_0^\infty r^{-2(\nu+\alpha)}e^{-2\lambda r}\,dr < \infty,
$$
since $\nu+\alpha<1/2$. This completes the proof.
\end{proof}
Based on Lemma \ref{lem:Znu-estimate}, we can obtain a result weaker than Theorem \ref{th2}, which also serves as a fundamental building block for the proof of Theorem \ref{th2}.
\begin{lemma}\label{lem:stochastic-estimate}
Suppose that $0<\alpha<1/2$ and $q>2/(1-2\alpha)$.  Then there exists a constant $C>0$, depending only on $\alpha,q,|\mathcal O|,\Theta,n,\lambda$, such that
$$
\mathbb{E}\big(\|A_q^\alpha u_n(t)\|_{L^q}^q\big) \le C(t+1),\qquad t\ge0.
$$
\end{lemma}

\begin{proof} It follows from \eqref{eq1.1} that for any $t\ge0$,
\begin{equation*}
\begin{split}
\mathbb{E}\big(\|A_q^\alpha u_n(t)\|_{L^q}^q\big)
\le&\,3^{q-1}\Big[\mathbb{E}\big(\|A_q^\alpha S(t\wedge\tau_n)u_0\|_{L^q}^q\big)
+\mathbb{E}\big(\|A_q^\alpha I_1(t)\|_{L^q}^q\big)
+\mathbb{E}\big(\|A_q^\alpha I_2(t)\|_{L^q}^q\big)\Big].
\end{split}
\end{equation*}

First, we have
\begin{equation}\label{eq2.2}
\mathbb{E}\big(\|A_q^\alpha S(t\wedge\tau_n)u_0\|_{L^q}^q\big)
\le C_{\alpha,q}\|A_q^\alpha u_0\|_{L^q}^q<\infty
\end{equation}
since $0<\alpha<1/2$ and $u_0\in D(A_0^{1/2})\hookrightarrow D(A_q^\alpha)$, and
\begin{equation}\label{eq2.5}
\begin{split}
\mathbb{E}\big(\|A_q^\alpha I_1(t)\|_{L^q}^q\big)
&\le \mathbb{E}\Bigg[\Bigg(\int_0^{t\wedge\tau_n}\|A_q^\alpha S(t\wedge\tau_n-s)f(u_n(s))\|_{L^q}\,ds\Bigg)^q\Bigg] \\
&\le |\mathcal O|\sup_{|u|\le n}|f(u)|^q
\left(\int_0^t(t-s)^{-\alpha}e^{-\lambda(t-s)}\,ds\right)^q \\
&\le C_{\alpha,q,n,|\mathcal O|,q,\a}<\infty.
\end{split}
\end{equation}

Choose $\nu\in(1/q,1/2-\alpha)$; such a choice is possible since $q>2/(1-2\alpha)$. By H\"older's inequality and the factorization formula,
$$
\begin{aligned}
\|A_q^\alpha I_2(t)\|_{L^q}^q
&\le \Bigg(\int_0^{t\wedge\tau}(t\wedge\tau-s)^{\nu-1}
\|S(t\wedge\tau-s)\|_{\mathcal L(L^q(\mathcal O))}
\|A_q^\alpha Z_\nu^n(s)\|_{L^q}\,ds\Bigg)^q \\
&\le C_q\Bigg(\int_0^{t\wedge\tau}\Big((t\wedge\tau-s)^{\nu-1}
e^{-\lambda(t\wedge\tau-s)}\Big)^{\frac{q}{q-1}}ds\Bigg)^{q-1}
\int_0^{t\wedge\tau}\|A_q^\alpha Z_\nu^n(s)\|_{L^q}^q\,ds \\
&\le C_{q,\nu}\int_0^t\|A_q^\alpha Z_\nu^n(s)\|_{L^q}^q\,ds,
\end{aligned}
$$
where the constant $C_{q,\nu}$ is finite because $\nu>1/q$, which implies $(\nu-1)q/(q-1)>-1$. Taking expectations and applying Lemma \ref{lem:Znu-estimate} yields
\be\label{eq3.12}
\mathbb{E}\big(\|A_q^\alpha I_2(t)\|_{L^q}^q\big)
\le C_{q,\nu}\int_0^t \mathbb{E}\|A_q^\alpha Z_\nu^n(s)\|_{L^q}^q\,ds
\le C_{q,|\mathcal O|,\Theta,n,\nu,\alpha,\lambda}\, t.
\ee

Combining \eqref{eq2.2}--\eqref{eq3.12}, we conclude that
$$
\mathbb{E}\big(\|A_q^\alpha u_n(t)\|_{L^q}^q\big) \le C_{q,|\mathcal O|,\Theta,n,\alpha,\lambda}(t+1),
$$
where $C>0$ is independent of $\nu$, since $\nu$ depends only on $\alpha$ and $q$. The proof is complete.
\end{proof}
\vs
To prove Theorem \ref{th2}, several lemmas are required.

\begin{lemma}\label{lem:Jn-estimate}
Suppose that $0<\a<1/2$, $0<\gam<1/2-\alpha$ and
$q>2/(1-2(\alpha+\gamma))$.
Then for any $T>0$, there exists a constant $C>0$, depending only on $\alpha,\gamma,q,|\mathcal O|,\Theta,n,\lambda$, such that for all $0\le t_1\le t_2\le T$,
\begin{equation}\label{eq2.6}
\mathbb{E}\big(\|A_q^\alpha [u_n(t_2)-u_n(t_1)]\|_{L^q}^q\big)
\le C(T+1)|t_2-t_1|^{q\gamma}.
\end{equation}
\end{lemma}

\begin{proof}

Choose $\nu$ satisfying
$$
\frac1q+\gamma<\nu<\frac12-\alpha.
$$
Such a choice is possible by the assumptions on $\alpha,\gamma,q$. 
\vs
Without loss of generality, assume $t_2\ge t_1\ge0$. Then
\begin{equation}\label{eq3.13}
\begin{split}
u_n(t_2)-u_n(t_1)
=&\,\big(S(t_2\wedge\tau_n-t_1\wedge\tau_n)-I\big)S(t_1\wedge\tau_n)u_0\\
&+\int_0^{t_1\wedge\tau_n}\big(S(t_2\wedge\tau_n-t_1\wedge\tau_n)-I\big)S(t_1\wedge\tau_n-s)f(u_n(s))\,ds\\
&+\int_{t_1\wedge\tau_n}^{t_2\wedge\tau_n} S(t_2\wedge\tau_n-s)f(u_n(s))\,ds\\
&+\int_0^{t_1\wedge\tau_n}\big(S(t_2\wedge\tau_n-t_1\wedge\tau_n)-I\big)S(t_1\wedge\tau_n-s)\sigma(u_n(s))\,dW(s)\\
&+\int_{t_1\wedge\tau_n}^{t_2\wedge\tau_n} S(t_2\wedge\tau_n-s)\sigma(u_n(s))\,dW(s)\\
=:&\sum_{i=1}^5 J_i.
\end{split}
\end{equation}

Using the smoothing estimate \eqref{eq2.11} and the modulus of continuity estimate \eqref{eq2.12}, we obtain
\begin{equation*}
\begin{split}
&\|A_q^\alpha J_1\|_{L^q}+\|A_q^\alpha J_2\|_{L^q}\\
\le&\,C_\gamma |t_2\wedge\tau_n-t_1\wedge\tau_n|^\gamma
\Big(\|A_q^{\alpha+\gamma}S(t_1\wedge\tau_n)u_0\|_{L^q}+\int_0^{t_1\wedge\tau_n}\|A_q^{\alpha+\gamma}S(t_1\wedge\tau_n-s)f(u_n(s))\|_{L^q}\,ds\Big)\\
\le&\,C_\gamma |t_2-t_1|^\gamma
\Big(\|A_q^{\alpha+\gamma}u_0\|_{L^q}+\int_0^{t_1\wedge\tau_n}\|A_q^{\alpha+\gamma}S(t_1\wedge\tau_n-s)\|_{\mathcal L(L^q)}\|f(u_n(s))\|_{L^q}\,ds\Big)\\
\le&\,C_{\gamma,\alpha}\sup_{|u|\le n}|f(u)|\,|\mathcal O|^{1/q}\,|t_2-t_1|^\gamma
\Big(\|A_q^{\alpha+\gamma}u_0\|_{L^q}+\int_0^{t_1\wedge\tau_n}(t_1\wedge\tau_n-s)^{-\alpha-\gamma}e^{-\lambda(t_1\wedge\tau_n-s)}\,ds\Big)\\
\le&\,C_{\gamma,\alpha,n,|\mathcal O|}\,|t_2-t_1|^\gamma
\big(\|A_q^{\alpha+\gamma}u_0\|_{L^q}+\lambda^{\alpha+\gamma-1}\Gamma(1-\alpha-\gamma)\big)\\
\le&\,C_{\gamma,\alpha,n,|\mathcal O|,\|A_q^{1/2}u_0\|_{L^q},\lambda}\,|t_2-t_1|^\gamma,
\end{split}
\end{equation*}
where, since $0<\alpha+\gamma<1/2$ and $u_0\in D(A_0^{1/2})\hookrightarrow D(A_q^{\alpha+\gamma})$, we have used the estimate
$
\|A_q^{\alpha+\gamma}u_0\|_{L^q}<\infty
$
in the last inequality.
\vs
For $J_3$, we have
\begin{equation*}
\begin{split}
\|A_q^\alpha J_3\|_{L^q} \le&\, C_{\alpha,n,|\mathcal O|}\int_{t_1\wedge\tau_n}^{t_2\wedge\tau_n}(t_2\wedge\tau_n-s)^{-\alpha}ds\\
\le&\,C_{\alpha,n,|\mathcal O|}|t_2\wedge\tau_n-t_1\wedge\tau_n|^{1-\alpha}\\
\le&\, C_{\alpha,n,|\mathcal O|}|t_2-t_1|^{1-\alpha}.
\end{split}
\end{equation*}

By the factorization formula \eqref{eq3.10}, we have
$$
J_4=\frac{\sin(\pi\nu)}{\pi}
\int_0^{t_1\wedge\tau_n}(t_1\wedge\tau_n-s)^{\nu-1}\big(S(h)-I\big)
S(t_1\wedge\tau_n-s)Z_\nu^n(s)\,ds
$$
and
$$
J_5=\frac{\sin(\pi\nu)}{\pi}
\int_{t_1\wedge\tau_n}^{t_2\wedge\tau_n}
(t_2\wedge\tau_n-s)^{\nu-1}S(t_2\wedge\tau_n-s)Z_\nu^n(s)\,ds,
$$
where $h:=t_2\wedge\tau_n-t_1\wedge\tau_n\ge0$.
\vs
For $J_4$, taking $A_q^\alpha$ and $L^q$-norms yields
$$
\|A_q^\alpha J_4\|_{L^q}
\le C_\nu
\int_0^{t_1\wedge\tau_n}(t_1\wedge\tau_n-s)^{\nu-1}
\|(S(h)-I)S(t_1\wedge\tau_n-s)A_q^\alpha Z_\nu^n(s)\|_{L^q}\,ds,
$$
where $\nu\in(1/q+\gamma,1/2-\alpha)$.
By the modulus of continuity estimate \eqref{eq2.12}, we have
$$
\|(S(h)-I)S(r)\|_{\mathcal L(L^q)}
\le C_\gamma h^\gamma r^{-\gamma} e^{-\lambda r},\qquad r>0.
$$
Applying this with $r=t_1\wedge\tau_n-s$, we get
$$
\|A_q^\alpha J_4\|_{L^q}
\le C_{\nu,\gamma} h^\gamma
\int_0^{t_1\wedge\tau_n}(t_1\wedge\tau_n-s)^{\nu-1-\gamma}
e^{-\lambda(t_1\wedge\tau_n-s)}
\|A_q^\alpha Z_\nu^n(s)\|_{L^q}\,ds.
$$
Then, by H\"older's inequality,
$$
\begin{aligned}
\|A_q^\alpha J_4\|_{L^q}
&\le C_{\nu,\gamma} h^\gamma
\Bigg(\int_0^{t_1\wedge\tau_n}
(t_1\wedge\tau_n-s)^{(\nu-1-\gamma)\frac{q}{q-1}}
e^{-\lambda(t_1\wedge\tau_n-s)\frac{q}{q-1}}\,ds\Bigg)^{\frac{q-1}{q}} \\
&\quad\times
\Bigg(\int_0^{t_1\wedge\tau_n}
\|A_q^\alpha Z_\nu^n(s)\|_{L^q}^q\,ds\Bigg)^{\frac{1}{q}}.
\end{aligned}
$$
The first integral is bounded by a constant $C_{\nu,\gamma,q,\lambda}$ because $\nu>1/q+\gamma$. Thus,
$$
\|A_q^\alpha J_4\|_{L^q}^q
\le C_{\nu,\gamma,q,\lambda} h^{q\gamma}
\int_0^{t_1\wedge\tau_n}
\|A_q^\alpha Z_\nu^n(s)\|_{L^q}^q\,ds.
$$
Taking expectations and applying Lemma \ref{lem:Znu-estimate}, we obtain
$$
\mathbb{E}\big(\|A_q^\alpha J_4\|_{L^q}^q\big)
\le C |t_2-t_1|^{q\gamma}\,t_1,
$$
where the constant $C>0$ depends on $\alpha,\nu,\gamma,q,|\mathcal O|,\Theta,n,\lambda$.
\vs
For $J_5$, taking $A_q^\alpha$ and $L^q$-norms yields
$$
\|A_q^\alpha J_5\|_{L^q}
\le C_\nu
\int_{t_1\wedge\tau_n}^{t_2\wedge\tau_n}
(t_2\wedge\tau_n-s)^{\nu-1}
\|S(t_2\wedge\tau_n-s)A_q^\alpha Z_\nu^n(s)\|_{L^q}\,ds.
$$
Using the smoothing estimate $\|S(r)\|_{\mathcal L(L^q)}\le C e^{-\lambda r}$ and dropping the exponential factor,
$$
\|A_q^\alpha J_5\|_{L^q}
\le C_{\nu}
\int_{t_1\wedge\tau_n}^{t_2\wedge\tau_n}
(t_2\wedge\tau_n-s)^{\nu-1}
\|A_q^\alpha Z_\nu^n(s)\|_{L^q}\,ds.
$$
Let $h:=t_2\wedge\tau_n-t_1\wedge\tau_n$. By H\"older's inequality,
$$
\begin{aligned}
\|A_q^\alpha J_5\|_{L^q}
&\le C_{\nu}
\Bigg(\int_{t_1\wedge\tau_n}^{t_2\wedge\tau_n}
(t_2\wedge\tau_n-s)^{(\nu-1)\frac{q}{q-1}}\,ds\Bigg)^{\frac{q-1}{q}} 
\Bigg(\int_{t_1\wedge\tau_n}^{t_2\wedge\tau_n}
\|A_q^\alpha Z_\nu^n(s)\|_{L^q}^q\,ds\Bigg)^{\frac{1}{q}}.
\end{aligned}
$$
Since $\nu>1/q$, the first integral is bounded by $C_{\nu,q} h^{\frac{\nu q-1}{q-1}}$. Hence,
$$
\|A_q^\alpha J_5\|_{L^q}^q
\le C_{\nu,q} h^{\nu q-1}
\int_{t_1\wedge\tau_n}^{t_2\wedge\tau_n}
\|A_q^\alpha Z_\nu^n(s)\|_{L^q}^q\,ds.
$$
Taking expectations and using $h\le |t_2-t_1|$,
$$
\mathbb{E}\big(\|A_q^\alpha J_5\|_{L^q}^q\big)
\le C_{\nu,q} |t_2-t_1|^{\nu q-1}
\mathbb{E}
\int_{t_1\wedge\tau_n}^{t_2\wedge\tau_n}
\|A_q^\alpha Z_\nu^n(s)\|_{L^q}^q\,ds.
$$
Since
$$
\mathbb{E}
\int_{t_1\wedge\tau_n}^{t_2\wedge\tau_n}
\|A_q^\alpha Z_\nu^n(s)\|_{L^q}^q\,ds
=
\int_{t_1}^{t_2}
\mathbb{E}\left[
\|A_q^\alpha Z_\nu^n(s)\|_{L^q}^q \mathbbm{1}_{\{s\le\tau_n\}}
\right] ds
$$
and $\mathbbm{1}_{\{s\le\tau_n\}}\le1$, Lemma \ref{lem:Znu-estimate} yields
$$
\mathbb{E}\int_{t_1\wedge\tau_n}^{t_2\wedge\tau_n}
\|A_q^\alpha Z_\nu^n(s)\|_{L^q}^q
 ds
\le
\int_{t_1}^{t_2} C\,ds
= C |t_2-t_1|.
$$
Therefore,
$$
\mathbb{E}\big(\|A_q^\alpha J_5\|_{L^q}^q\big)
\le C |t_2-t_1|^{\nu q},
$$
where $C>0$ depends on $\alpha,\nu,q,|\mathcal O|,\Theta,n,\lambda$.
\vs
Summarizing the estimates for $\mathbb{E}\|A_q^\alpha J_i\|_{L^q}^q$, $i=1,\cdots,5$, and noting that $\gamma < \min\{1-\alpha,\nu\}$ and $\nu$ can be chosen to depend only on $\gamma,\alpha$ and $q$, we obtain \eqref{eq2.6}.
\end{proof}
\vs

The following result is builded upon Lemma \ref{lem:Jn-estimate} and Proposition \ref{p1}.
\begin{corollary}\label{c2}
Suppose that $0<\a<1/2$, $0<\gam<1/2-\a$ and$$
q>\max\left\{\frac{2}{1-2(\alpha+\gamma)},\,\frac{1}{\gamma}\right\}.
$$
Then for any $T>0$ and $0<\eta<\gamma-1/q$, there exists  a random variable $K(\omega)>0$ such that for a.e. $\omega\in\Omega$,
$$
\|A_q^{\alpha}[u_n(t_1)-u_n(t_2)]\|_{L^q} \le K(\omega)|t_1-t_2|^\eta, \qquad 0\le t_1,t_2\le T,
$$
and there exists a constant $B>0$, depending on $\|A_q^{1/2}u_0\|_{L^q}$, $\alpha,\gamma,q,T,|\mathcal O|,\Theta,n,\lambda$, such that
$$
\mathbb{E}(K^q) \le B.
$$
\end{corollary}


\begin{lemma}\label{le3.10}
For any $d\in\mathbb N^+$, there exist real numbers $\alpha,\gamma$ satisfying
\be\label{eq3.2}
\begin{aligned}
& 0 < \alpha < \frac{1}{2},\quad 0 < \gamma < \min\left\{\a,\frac{1}{4} - \frac{\alpha}{2}\right\}, \\
& q>\max\left\{\frac{2}{1-2(\alpha+2\gamma)},\,\frac{1}{2\a},\,\frac{1}{\gamma}\right\},
\end{aligned}
\ee
if and only if $$q>8.$$
\end{lemma}
We have provided a proof of Lemma \ref{le3.10} in Appendix \ref{A.4}.
\Vs
The following lemma plays a key role in the proof of Theorem \ref{th2}.
\begin{lemma}\label{le2.6} 
Let
$$
I_{21}(t):=\int_0^{t\wedge\tau_n}S(t\wedge\tau_n-s)[\sigma(u_n(s))-\sigma(0)]\,dW(s),\qquad t\ge0.
$$
Assume that $q>8$.
Then for any $T>0$, there exists a constant $C>0$, depending on $\|A_q^{1/2}u_0\|_{L^q}$, $q$, $d$, $n$, $|\mathcal O|$, $\Theta$, $\lambda$, and $T$, such that
$$
\mathbb{E}\|A_q^{1/2}I_{21}(t)\|_{L^q}^q \le C,\qquad 0\le t\le T.
$$
\end{lemma}

\begin{proof} 
For $q>8$, let $\alpha,\gamma$ be the parameters satisfying \eqref{eq3.2}. Under the above assumptions, one can readily choose a real number $\nu$ satisfying
$$\frac{1}{q}+\gam<\nu<\frac{1}{2}-\a,\quad 0<\a-\nu<\frac{1}{2}.$$
\vs
By the factorization formula, 
$$
I_{21}(t)=\frac{\sin(\pi\nu)}{\pi}\int_0^{t\wedge\tau_n}(t\wedge\tau_n-r)^{\nu-1}S(t\wedge\tau_n-r)Z_\nu^{n1}(r)\,dr,
$$
where
$$
Z_\nu^{n1}(r):=\int_0^r(r-s)^{-\nu}S(r-s)[\sigma(u_n(s))-\sigma(u_n(0))]\mathbbm{1}_{\{s\le\tau_n\}}\,dW(s).
$$

Denote $\Psi(t):=\sigma(t)-\sigma(0)$. For $0<\a'<\a<1/(2q)$, we can apply Proposition \ref{prop:sigma-W2aq} and Corollary \ref{c2} to obtain, for a.e. $\w\in\W$,
$$\begin{aligned}
\|A_q^{\a'}\Psi(u_n(s))\|_{L^q}\le& CL_n\|A_q^\a u_n(s)\|_{L^q}\\
\le&CL_n\(\|A_q^\a(u_n(s)-u_n(0))\|_{L^q}+\|A_q^{\a}u_n(0)\|_{L^q}\)\\
\le& CL_n\big(K(\omega)s^\eta+\|A_q^\a u_0\|_{L^q}\big)=:\~K(\w),\end{aligned}
$$ where $L_n:=\sup_{|t|\le n}|\sig'(t)|$.
It is clear that $\E(\~K^q)<\8$.
By the BDG inequality, we have
$$
\begin{aligned}
&\mathbb{E}\big(\|A_q^{1/2}Z_\nu^{n1}(r)\|_{L^q}^q\big)\\
\le& C_{q,|\mathcal O|,\Theta}\,
\mathbb{E}\Bigg[\Bigg(
\int_0^r(r-s)^{-2\nu}(r-s)^{-(1-2\alpha')}e^{-2\lambda(r-s)}
\|A_q^{\a'}\Psi(u_n(s))\|_{L^q}^2\,ds
\Bigg)^{q/2}\Bigg] \\
\le& C_{q,|\mathcal O|,\Theta,n}\,\mathbb{E}(\tilde K^q)\,
\left[
\int_0^r(r-s)^{-1+2(\alpha'-\nu)}e^{-2\lambda(r-s)}s^{2\eta}\,ds
\right]^{q/2}\\
\le& C_{q,|\mathcal O|,\Theta,n}\,\mathbb{E}(\tilde K^q)\,t^{q\eta}
\left[
\int_0^r(r-s)^{-1+2(\alpha'-\nu)}e^{-2\lambda(r-s)}\,ds
\right]^{q/2}.
\end{aligned}
$$

Since $0<\alpha-\nu<1/2$, we may choose $\alpha'$ sufficiently close to $\alpha$ so that $0<\alpha'-\nu<1/2$ still holds. Then the integral
$
\int_0^r(r-s)^{-1+2(\alpha'-\nu)}\,e^{-2\lambda(r-s)}ds
$
converges uniformly for $r\le T$. Hence
$$
\mathbb{E}\big(\|A_q^{1/2}Z_\nu^{n1}(r)\|_{L^q}^q\big)
\le C_{q,|\mathcal O|,\Theta,n,\nu,\alpha,\eta,\lambda,T}.
$$

Taking $A_q^{1/2}$ and $L^q$-norms in the factorization formula, and using $\|S(\rho)\|_{\mathcal L(L^q)}\le C e^{-\lambda\rho}$, we get
$$
\|A_q^{1/2}I_{21}(t)\|_{L^q}
\le C_\nu
\int_0^{t\wedge\tau_n}(t\wedge\tau_n-r)^{\nu-1}
e^{-\lambda(t\wedge\tau_n-r)}
\|A_q^{1/2}Z_\nu^{n1}(r)\|_{L^q}\,dr.
$$
By H\"older's inequality,
$$\begin{aligned}\|A_q^{1/2}I_{21}(t)\|_{L^q} \le& C_\nu \int_0^{t \wedge \tau_n} (t \wedge \tau_n - r)^{\nu-1} e^{-\lambda(t \wedge \tau_n-r)} \|A_q^{1/2}Z_{\nu}^{n1}(r)\|_{L^q} \, dr\\
\le& C_\nu
\left( \int_0^{t \wedge \tau_n} (t \wedge \tau_n - r)^{(\nu-1)\frac{q}{q-1}} e^{-\lambda(t \wedge \tau_n-r)\frac{q}{q-1}} \, dr \right)^{\frac{q-1}{q}} \\
&\quad\times
\left( \int_0^{t \wedge \tau_n} \|A_q^{1/2}Z_{\nu}^{n1}(r)\|_{L^q}^q \, dr \right)^{\frac{1}{q}}.
\end{aligned}$$
Since $-1<(\nu-1)\frac{q}{q-1}<0$ (which follows from $\nu\in(1/q+\gamma,1/2-\alpha)$), we have $$\int_0^{t \wedge \tau_n} (t \wedge \tau_n - r)^{(\nu-1)\frac{q}{q-1}} e^{-\lambda(t \wedge \tau_n-r)\frac{q}{q-1}} \, dr\le \int_0^\infty r^{(\nu-1)\frac{q}{q-1}}e^{-2\lambda r\frac{q}{q-1}}\,dr < \infty.$$
It follows that $$
\|A_q^{1/2}I_{21}(t)\|_{L^q}^q
\le C_{\nu,q,\lambda}
\int_0^{t\wedge\tau_n}\|A_q^{1/2}Z_\nu^{n1}(r)\|_{L^q}^q\,dr.
$$
Taking expectations and using the estimate for $Z_\nu^{n1}$, we obtain
$$
\mathbb{E}\big(\|A_q^{1/2}I_{21}(t)\|_{L^q}^q\big)
\le C_{\nu,q,\lambda}\int_0^t C\,dr \le C.
$$
This completes the proof.
\end{proof}

\begin{lemma}\label{le2.8}
Suppose that $\sigma(0)\neq0$. Let
$$
I_{22}(t):=\int_0^{t\wedge\tau_n}S(t\wedge\tau_n-s)\sigma(0)\,dW(s),\qquad t\ge0.
$$
Assume that $q>2/\delta$. Then
$$
\mathbb{E}\big(\|A_q^{1/2}I_{22}(t)\|_{L^q}^q\big) \le C,\qquad t\ge0,
$$
where $C>0$ depends on $q$, $|\mathcal O|$, $\Theta'$, $\delta$ and $\lambda$.
\end{lemma}

\begin{proof}
For simplicity, we assume $\sigma(0)=1$. By the factorization formula (cf. \eqref{eq3.10}), we have
$$
I_{22}(t)=\frac{\sin(\pi\rho)}{\pi}\int_0^{t\wedge\tau_n}(t\wedge\tau_n-r)^{\rho-1}S(t\wedge\tau_n-r)Z_\rho^{n2}(r)\,dr,
$$
where
$$
Z_\rho^{n2}(r):=\int_0^r(r-s)^{-\rho}S(r-s)\mathbbm{1}_{\{s\le\tau_n\}}\,dW(s).
$$

We choose $\rho$ such that
$$
\frac1q<\rho<\frac{\delta}{2}.
$$
Such a choice is possible since $q>2/\delta$.

By the BDG inequality and the smoothing estimate
$\|A_b^{(1-\delta)/2}S(r)\|_{\mathcal L(C_b)}\le C r^{-(1-\delta)/2}$,
\begin{align*}
A_q^{1/2}Z_\rho^{n2}(r) 
&= \sum_{j=1}^\infty\sqrt{\mu_j}\int_0^r(r-s)^{-\rho} A_q^{1/2}S(r-s)e_j \,dB_j(s)\\
&=\sum_{j=1}^\infty\sqrt{\mu_j}\lambda_j^{\delta/2}\int_0^r(r-s)^{-\rho} A_b^{(1-\delta)/2}S(r-s)e_j \,dB_j(s),
\end{align*}
where the fact $A_b^{\delta/2}e_j=\lambda_j^{\delta/2}e_j$, $j\in\mathbb Z^+$ has been used.

Applying the Burkholder--Davis--Gundy inequality and using $\Theta':=\sum_{j=1}^\infty\lambda_j^\delta\mu_j\|e_j\|_{C_b}^2<\infty$, we obtain
\begin{align*}
\mathbb{E}\big(\|A_q^{1/2}Z_\rho^{n2}(r)\|_{L^q}^q\big)
&\le C_{q,|\mathcal O|}\,\mathbb{E}\Bigg[\Bigg(
\sum_{j=1}^\infty\lambda_j^\delta\mu_j\int_0^r(r-s)^{-2\rho}
\|A_b^{(1-\delta)/2}S(r-s)e_j\|_{C_b}^2\,ds
\Bigg)^{q/2}\Bigg] \\
&\le C_{q,|\mathcal O|}\,
\Bigg(\sum_{j=1}^{\infty}\lambda_j^\delta\mu_j\|e_j\|_{C_b}^2\Bigg)^{q/2}
\Bigg(\int_0^r(r-s)^{\delta-2\rho-1}e^{-\lambda(r-s)}\,ds\Bigg)^{q/2}\\
&\le C_{q,|\mathcal O|,\Theta'}\,
\Bigg(\int_0^r(r-s)^{\delta-2\rho-1}e^{-\lambda(r-s)}\,ds\Bigg)^{q/2}.
\end{align*}
Since $\rho<\delta/2$, the integral
$$
\int_0^r(r-s)^{\delta-2\rho-1}e^{-\lambda(r-s)}\,ds
\le \int_0^\infty r^{\delta-2\rho-1}e^{-\lambda r}\,dr
= \lambda^{-(\delta-2\rho)}\Gamma(\delta-2\rho)<\infty.
$$Hence
$$
\mathbb{E}\big(\|A_q^{1/2}Z_\rho^{n2}(r)\|_{L^q}^q\big)
\le C_{q,|\mathcal O|,\Theta',\rho,\delta,\lambda}.
$$

Taking $A_q^{1/2}$ and $L^q$-norms in the factorization formula, and using $\|S(\rho)\|_{\mathcal L(L^q)}\le C e^{-\lambda\rho}$, we get
$$
\|A_q^{1/2}I_{22}(t)\|_{L^q}
\le C_\rho
\int_0^{t\wedge\tau_n}(t\wedge\tau_n-r)^{\rho-1}
e^{-\lambda(t\wedge\tau_n-r)}
\|A_q^{1/2}Z_\rho^{n2}(r)\|_{L^q}\,dr.
$$
By H\"older's inequality,
\begin{align*}
\|A_q^{1/2}I_{22}(t)\|_{L^q}
&\le C_\rho
\left( \int_0^{t\wedge\tau_n} (t\wedge\tau_n-r)^{(\rho-1)\frac{q}{q-1}} e^{-\lambda(t\wedge\tau_n-r)\frac{q}{q-1}} \, dr \right)^{\frac{q-1}{q}} \\
&\quad\times
\left( \int_0^{t\wedge\tau_n} \|A_q^{1/2}Z_\rho^{n2}(r)\|_{L^q}^q \, dr \right)^{\frac{1}{q}}.
\end{align*}
Since $\rho>1/q$, the first integral is bounded by a constant $C_{\rho,q,\lambda}$. Thus
$$
\|A_q^{1/2}I_{22}(t)\|_{L^q}^q
\le C_{\rho,q,\lambda}
\int_0^{t\wedge\tau_n}\|A_q^{1/2}Z_\rho^{n2}(r)\|_{L^q}^q\,dr.
$$
Taking expectations and using (3.18), we obtain
$$
\mathbb{E}\big(\|A_q^{1/2}I_{22}(t)\|_{L^q}^q\big)
\le C_{\rho,q,\lambda}\int_0^t C_{q,|\mathcal O|,\Theta',\rho,\delta,\lambda}\,dr
\le C_{q,|\mathcal O|,\Theta',\rho,\delta,\lambda}.
$$
This completes the proof.
\end{proof}

\begin{proof}[Proof of Theorem \ref{th2}]
We first prove the result for the case $q>q_0:=\max\{8,2/\delta\}$ (and $q>2/\delta$ if $\sigma(0)\neq0$).
By decomposing $\sigma(u_n(s))=\big[\sigma(u_n(s))-\sigma(0)\big]+\sigma(0)$,
we write the stochastic convolution $I_2(t)=I_{21}(t)+I_{22}(t)$, where
$$
I_{21}(t):=\int_0^{t\wedge\tau_n}S(t\wedge\tau_n-s)\big[\sigma(u_n(s))-\sigma(0)\big]\,dW(s),
$$
$$
I_{22}(t):=\int_0^{t\wedge\tau_n}S(t\wedge\tau_n-s)\sigma(0)\,dW(s).
$$
Applying Lemmas \ref{le2.6} and \ref{le2.8} with exponent $q$, we obtain
$$
\mathbb{E}\big(\|A_{q}^{1/2}I_2(t)\|_{L^{q}}^{q}\big)\le2^{q-1}\(\mathbb{E}\|A_{q}^{1/2}I_{21}(t)\|_{L^{q}}^{q}+
\mathbb{E}\|A_{q}^{1/2}I_{22}(t)\|_{L^{q}}^{q}\)\le C.
$$

If $1<q\le q_0$, since $\mathcal O$ is bounded, we have the continuous embedding
$W_0^{1,q_0+1}(\mathcal O)\hookrightarrow W_0^{1,q}(\mathcal O)$.
Recall the norm equivalence
$\|A_r^{1/2}v\|_{L^r}\asymp \|v\|_{W_0^{1,r}(\mathcal O)}$ for $r\ge2$.
For each sample path,
$$
\|I_2(t)\|_{W_0^{1,q}} \le C_{q,q_0}\|I_2(t)\|_{W_0^{1,q_0+1}}.
$$
Taking $q$-th power, expectation and applying the inequality
$\mathbb{E} \big(X^q\big) \le \bigl(\mathbb{E} X^{q_0+1} \bigr)^{q/(q_0+1)}$, we get
$$
\begin{aligned}
\mathbb{E}\big(\|A_q^{1/2}I_2(t)\|_{L^q}^q\big)
&\le C\,\mathbb{E}\big(\|I_2(t)\|_{W_0^{1,q_0+1}}^q\big) \le C \Big[\mathbb{E}\Big(\|I_2(t)\|_{W_0^{1,q_0+1}}^{q_0+1}\Big)\Big]^{\frac{q}{q_0+1}}
\le C.
\end{aligned}
$$

On the other hand, for the deterministic integral part, it is easy to verify that
$$
\mathbb{E}\big(\|A_q^{1/2}I_1(t)\|_{L^q}^q\big)
\le C_{\|A_q^{1/2}u_0\|_{L^q},q,n,|\mathcal O|,\lambda}.
$$

Finally, combining the estimates for $I_1$ and $I_2$, we conclude that
\begin{align*}
\mathbb{E}\big(\|A_q^{1/2}u_n(t)\|_{L^q}^q\big)
&\le 2^{q-1}\Big[\mathbb{E}\big(\|A_q^{1/2}I_1(t)\|_{L^q}^q\big)
+\mathbb{E}\big(\|A_q^{1/2}I_2(t)\|_{L^q}^q\big)\Big] \\
&\le C,
\end{align*}
where $C>0$ depends on $\|A_q^{1/2}u_0\|_{L^q}$, $q$, $d$, $n$, $|\mathcal O|$, $\Theta$, $\Theta'$, $\delta$ and $\lambda$, and is independent of $\Theta'$ and $\de$ whenever $\sigma(0)=0$.
The proof is complete.
\end{proof}\section{Global existence and mean dissipativity of mild solutions}\label{se4}

\subsection{Assumptions and main result}
This section is dedicated to proving  the global existence and mean dissipativity of mild solutions to equation \eqref{eq}  (obtained in Proposition \ref{th1}) under some additional assumptions on $f$ and $\sig$.

\hypertarget{H}{}\paragraph{(H2)} \label{H2}
(Dissipativity and polynomial growth conditions) There exist $q>d+6$ and $r\ge1$ such that
\begin{equation}\label{eq3.9}
f(u)u +(qr-1)\Theta\,|\sigma(u)|^2 \le -c_1 u^2 + M_{q,r}, \Hs  u \in \mathbb{R}
\end{equation}
and 
\be\label{eq4.16}|f(u)|\vee|\sig(u)|\le c_2(|u|^r+1),\Hs u\in\R,\ee
for some constants $c_1,c_2 > 0$, where $
\Theta = \sum_{j=1}^{\infty} \mu_j \|e_j\|_{C_0}^2<\8$ and $M_{q,r}>0$ is a constant depending on $q$ and $r$.
\Vs
\hypertarget{H}{}\paragraph{(H3)} \label{H3}
(Polynomial Lipschitz and one-sided conditions) There exist constants $c_3 > 0$ such that
\begin{equation}\label{eq3.5}
|f(u)-f(v)| \vee |\sigma(u)-\sigma(v)| \le c_3\big(1 + |u|^{r-1}+|v|^{r-1}\big)|u-v|, \Hs  u,v \in \mathbb{R},
\end{equation}
\be\label{eq4.5}
(f(u)-f(v))(u-v) + (qr-1)\Theta\,|\sigma(u)-\sigma(v)|^2 \le N_{q,r} |u-v|^2 , \Hs  u,v \in \mathbb{R},
\ee where $q$ and $r\geq1$ are the same as that in \hyperref[H2]{(H2)} and $N_{q,r}>0$ is a constant depending on $q$ and $r$.

\vs

\begin{example}\label{ex1}
Assume that $\chi\ge1$ and $\beta+1> 2\chi$. A canonical example of functions $f,\sigma$ satisfying conditions \hyperref[H2]{(H2)} and \hyperref[H3]{(H3)} is given by
$$
f(u) = \sum_{j=1}^{\beta} b_j u|u|^{j-1}, \Hs \text{with } b_\beta < 0,
$$
and $\sigma$ being a polynomial satisfying
$$
|\sigma(u)| \leq c(|u|^\chi + 1).
$$It is easy to verify that for {\sl any} $q \ge2$ and $r = \beta$,  $f$ and $\sigma$ satisfy \hyperref[H2]{(H2)} and \hyperref[H3]{(H3)}.
\end{example}
The main result of the paper is given as follows.
\begin{theorem}\label{th}
Assume that \hyperref[H1]{(H1)}--\hyperref[H3]{(H3)} hold and that $u_0 \in C_0(\overline{\cO})$. 
Then the $C_0(\overline{\mathcal{O}})$-valued mild solution of equation \eqref{eq1} exists globally in time, i.e., $\tau(\w)=\8$ for a.e. $\w\in\W$. Moreover, for any $p\le q$, the mild solution $u$ exhibits mean dissipativity in the sense that
\begin{equation}\label{eq4.1}
\mathbb{E}\| u(t) \|_{C_b}^p \le C \big(\| u_0 \|_{C_b}^{pr} e^{-\varrho_{pr}(t-1)} + 1 \big), \Hs t \ge 1,
\end{equation}
where $\varrho_{pr}:=pr\min\{1/2,c_1/4\}>0$  and $C> 0$ depends on $p$, $d$, $r$, $|\mathcal{O}|$ and $\Theta$.
\end{theorem}
\begin{remark}
Strictly speaking, the norm denoted by $\|\cdot\|_{C_b}$ in  \eqref{eq4.1} is the norm on $C_0(\overline{\cO})$, since $u(t)\in C_0(\overline{\cO}) $, and thus should be written as $\|\.\|_{C_0}$. However, since the two norms coincide on $C_0(\overline{\cO})$ and both spaces will appear in the proof, for simplicity we continue to denote it uniformly by $\|\cdot\|_{C_b}$.
\end{remark}





\subsection{An auxiliary result under stronger assumptions}
As a preparation for the proof of Theorem \ref{th}, we first establish the following result under stronger assumptions on the initial data and the noise. This auxiliary result, combined with \hyperref[H3]{(H3)}, will then allow us to obtain Theorem \ref{th} via an approximation argument.

\begin{theorem}\label{th:global}
Assume \hyperref[H1*]{(H1*)} and \hyperref[H2]{(H2)} hold.
Then for any $u_0 \in D(A_0^{1/2})$, the conclusions of Theorem \ref{th} follow.
\end{theorem}

Before proving Theorem \ref{th:global}, we make some preliminary observations.

\begin{lemma}\label{le4.2} Let $u_n(t) := u(t \wedge \tau_n)$, $t \ge 0$, be the corresponding stopped process, where $\tau_n := \inf\{t \ge 0 : \|u(t)\|_{C_b} \ge n\}$.
 Assume \hyperref[H1*]{(H1*)} and  the dissipativity condition \eqref{eq3.9} hold. 
Then there exist constants $\~R_{qr},R_{qr}>0$ such that 
\begin{equation}\label{eq3.10'}
\mathbb{E}\big(\|u_n(t)\|_{L^{qr}}^{qr}\big)+\rho_{qr}\int_0^t \E\left(\mathbbm{1}_{\{s\le\tau_n\}}\|u_n(s)\|_{L^{qr}}^{qr}\right)ds
\le \|u_0\|_{L^{qr}}^{qr}
+\~R_{qr},\Hs t\ge0,
\end{equation}
\begin{equation}\label{eq:Lq-bound}
\mathbb{E}\left(\mathbbm{1}_{\{t\le\tau_n\}}\|u_n(t)\|_{L^{qr}}^{qr}\right)
\le e^{-\rho_{qr}t}\|u_0\|_{L^{qr}}^{qr}
+R_{qr},\Hs t\ge0,
\end{equation}where $\rho_{qr}:=c_1 qr/2$.
\end{lemma}


\begin{proof} For simplicity, denote $\vartheta:=qr$.
In view of the identification $D(A_\vartheta^{1/2})=W_0^{1,\vartheta}(\mathcal O)$, it follows from Theorem \ref{th2} that
$$
\mathbb{E}\int_0^T \|u_n(t)\|_{W_0^{1,\vartheta}(\mathcal O)}^\vartheta\,dt < \infty,\qquad \forall\, T>0.
$$
By the equivalence of mild and weak solutions for analytic semigroups (see e.g., \cite[Theorem 6.5]{DZ}), $u_n$ also satisfies the weak formulation: for any $\varphi \in C_0^\infty(\mathcal O)$,
\begin{equation*}
\begin{aligned}
(u_n(t), \varphi) =&\, (u_0, \varphi) 
                     - \int_0^t \mathbbm{1}_{\{s \le\tau_n\}} (\nabla u_n(s),\nabla\va)\, ds \\
                   &+ \int_0^t \mathbbm{1}_{\{s \le\tau_n\}} (f(u_n(s)), \varphi) \, ds \\
                   &+ \int_0^t \mathbbm{1}_{\{s \le\tau_n\}} (\sigma(u_n(s)), \varphi) \, dW(s), \qquad 0 \le t \le T,
\end{aligned}
\end{equation*}
where $(\cdot,\cdot)$ denotes the inner product in $L^2(\mathcal O)$.

The $L^p$-norm It\^o formula for SPDEs on bounded domains with Dirichlet boundary conditions is well-established in the literature for equations with linear multiplicative noise (see e.g., \cite[Lemma 3.3]{DG15}). Following the same derivation, for our equation with nonlinear diffusion coefficient $\sigma(u_n)$, we obtain the following formula.

Applying the $L^q$-norm It\^o formula to $u_n(t)$ and taking expectations, we obtain
$$
\begin{aligned}
&\mathbb{E}\big(\|u_n(t)\|_{L^\vartheta}^\vartheta\big)
+\vartheta(\vartheta-1)\mathbb{E}\int_0^t\mathbbm{1}_{\{s\le\tau_n\}}\int_{\mathcal O}|u_n|^{\vartheta-2}|\nabla u_n|^2 dxds\\
&=\mathbb{E}\|u_0\|_{L^\vartheta}^\vartheta
+\vartheta\mathbb{E}\int_0^t\mathbbm{1}_{\{s\le\tau_n\}}\int_{\mathcal O}|u_n|^{\vartheta-2}u_n f(u_n)dxds\\
&\quad+\frac{\vartheta(\vartheta-1)}{2}\mathbb{E}\int_0^t\mathbbm{1}_{\{s\le\tau_n\}}\int_{\mathcal O}|u_n|^{\vartheta-2}|\sigma(u_n)|^2 \sum_{j=1}^{\infty} \mu_j e_j^2dxds,
\end{aligned}
$$
where the stochastic integral term vanishes after taking expectations.
\vs
From the dissipativity condition \eqref{eq3.9} and Young's inequality, we have
$$
\begin{aligned}
&\,\vartheta|u|^{\vartheta-2}u f(u)+\frac{\vartheta(\vartheta-1)}{2}|u|^{\vartheta-2}|\sigma(u)|^2\sum_{j=1}^{\infty} \mu_j e_j^2\\
\le&\, \vartheta|u|^{\vartheta-2}\left[u f(u)+(\vartheta-1)\Theta|\sigma(u)|^2\right]\\
\le&\, \vartheta|u|^{\vartheta-2}(-c_1|u|^2+M_\vartheta)\\
\le&\, -\frac{\vartheta c_1}{2}|u|^\vartheta+C_{M_\vartheta,c_1,\vartheta},
\end{aligned}
$$ where $$C_{M_\vartheta,c_1,\vartheta}:=\frac{2\sqrt{2}}{\vartheta^{3/2}\sqrt{c_1}} M_\vartheta^{\frac{\vartheta}{2}}.$$
Substituting this into the It\^o formula above yields
$$
\mathbb{E}\big(\|u_n(t)\|_{L^\vartheta}^\vartheta\big)
\le\|u_0\|_{L^\vartheta}^\vartheta
-\frac{c_1 \vartheta}{2}\int_0^t \E\left[\mathbbm{1}_{\{s\le\tau_n\}}\|u_n(s)\|_{L^\vartheta}^\vartheta\right]ds
+C_{M_\vartheta,c_1,\vartheta}\,t.
$$
By Gronwall's inequality, we obtain
$$
\mathbb{E}\left(\mathbbm{1}_{\{t\le\tau_n\}}\|u_n(t)\|_{L^\vartheta}^\vartheta\right)
\le e^{-\frac{c_1 \vartheta}{2}t} \|u_0\|_{L^\vartheta}^\vartheta
+\frac{2C_{M_\vartheta,c_1,\vartheta}}{c_1 \vartheta},
$$ which completes the proof.
\end{proof}

\begin{lemma}\label{le4.3}
For any $d\in\mathbb N^+$, there exist parameters $\nu,\gamma\in(0,1/2)$ satisfying
$$
\begin{cases}
0<\frac{d}{2(q-1)}<1,\\[0.3em]
0<\frac{2\nu q}{q-2}<1,\\[0.3em]
0<\left(1+\frac{d}{2q}-\nu+\gam\right)\frac{q}{q-1}<1,\\[0.3em]
0<(\gamma+\frac{d}{2q})\frac{q}{q-1}<1,\\[0.3em]
q\gamma>1,\\[0.3em]
q\nu>\frac{d}{2},
\end{cases}
$$
if and only if $$q>d+6.$$ 
\end{lemma}
We have provided a proof of Lemma \ref{le4.3} in Appendix \ref{A.4}.
\Vs

As a result of Lemma~\ref{le4.3}, the following integrals are convergent.
\begin{align*}
\Gamma_1:=\int_0^{t} (t - s)^{-\frac{d}{2(q-1)}} \,
        e^{-\frac{\lambda (t - s)}{2} \cdot \frac{q}{q-1}} \, ds<\8;
\end{align*}
\begin{align*}
\Gamma_2:=\int_0^t(t-s)^{-\frac{2\nu q}{q-2}}e^{-2\lambda(1-\ve)\frac{q}{q-2}(t-s)}ds<\8,
\end{align*}where $0<\ve<1$;
\begin{align*}\Gamma_3:=\int_0^{t}(t-s)^{-(1+\frac{d}{2q}-\nu)\frac{q}{q-1}}e^{-\frac{\lam}{2}\frac{q}{q-1}(t-s)}ds<\8;
\end{align*}
\begin{align*}
\Gamma_4:=\int_0^t(t-s)^{-(\gamma+\frac{d}{2q})\frac{q}{q-1}}e^{-\lambda\frac{q}{q-1}\frac{t-s}{2}}ds
              <\8;
\end{align*}
\begin{align*}
\Gamma_5:=\int_0^t(t-s)^{-(1+\frac{d}{2q}-\nu+\gam)\frac{q}{q-1}}e^{-\lam\frac{(t-s)q}{2{q-1}}}ds<\8.
\end{align*}

\vs

\bl Assume \hyperref[H1*]{(H1*)} and \hyperref[H2]{(H2)} hold. Then for any $u_0 \in D(A_0^{1/2})$, the mild solution $u$ is global in time, i.e., for a.e. $\w\in\W$, $\tau(\w)=\8$. \el
 \begin{proof}Let $u_n(t) := u(t \wedge \tau_n)$, $t \ge 0$, be the corresponding stopped process, where $\tau_n := \inf\{t \ge 0 : \|u(t)\|_{C_0} \ge n\}$. To establish the global existence of $u(t)$ in $C_b(\mathcal{O})$ for a.e. $\omega\in\W$, it suffices to show that $\tau_n \to \infty$ as $n \to \infty$ almost surely. 
\vs
Thanks to Lemma \ref{le4.2}, we first establish estimates for each term in \eqref{eq1.1} and \eqref{eq3.13} that are uniform in $n$. 

We shall make use of the following smoothing estimate for the semigroup, which follows from \cite[Proposition 48.4* (e)]{QS}:
\begin{equation}\label{eq3.4}
\begin{split}
\|S(t) u \|_{C_b}
\le C \|S(t/2)\|_{\cL(C_b)}\|S(t/2) u \|_{C_b}\le C_{q,d}\, e^{-\frac{\lam t}{2}} t^{-\frac{d}{2q}} \| u \|_{L^q}, \quad t\ge0,\; u \in L^q(\mathcal{O}).
\end{split}
\end{equation}

Combining \eqref{eq3.4} with the growth condition \eqref{eq4.16}, we obtain
\begin{equation*}
\begin{split}
\big\| S(t \wedge\tau_n - s) f(u_n(s)) \big\|_{C_b}
\le&\, C_{q,d}\, e^{-\frac{\lambda (t \wedge\tau_n - s)}{2}} (t \wedge\tau_n - s)^{-\frac{d}{2q}} \| f(u_n(s)) \|_{L^q} \\
\le&\, C_{q,d}\, e^{-\frac{\lambda (t \wedge\tau_n - s)}{2}} (t \wedge\tau_n - s)^{-\frac{d}{2q}}
        \big( \| u_n(s) \|_{L^{qr}}^{r} + 1 \big).
\end{split}
\end{equation*} 
Then, by H\"older's inequality,
\begin{equation}\label{eq4.2}\begin{split}
      \Big[\int_0^{t \wedge \tau_n} \big\|S(t \wedge \tau_n - s) f(u_n(s)) \big\|_{C_b} \, ds\Big]^{q} 
\le&\, C_{q,d} \Gamma_1^{q-1}
         \Big(\int_0^t \mathbbm{1}_{\{s \le\tau_n\}} \| u_n(s) \|_{L^{qr}}^{qr} \, ds +1 \Big)\\
\le&\, C_{q,d,\lambda} \Big[ \Big( \int_0^t \mathbbm{1}_{\{s \le\tau_n\}} \| u_n(s) \|_{L^{qr}}^{qr} \, ds \Big)^{\frac{1}{q}} + 1 \Big],
\end{split}
\end{equation}
where $C_{q,d,\lambda}$ depends on $q$, $d$, $\lambda$ (via $\Gamma_1$, $\Gamma_2$ and $C_{q,d}$).
Consequently,
\begin{equation*}\begin{split}\label{eq4.10}
\mathbb{E}\big(\| I_1(t) \|_{C_b}^q\big) \le&\, 
          C_{q,d,\lambda} \Big( \mathbb{E}\int_0^t\mathbbm{1}_{\{s \le\tau_n\}} \| u_n(s) \|_{L^{qr}}^{qr} ds + 1 \Big)\\
\stackrel{\eqref{eq3.10'}}{\le}&\, C_{q,d,\lambda,r} \big(\| u_0 \|_{L^{qr}}^{qr} + t + 1 \big).\end{split}
\end{equation*}
\vs

Recall the factorization formula for $I_2(t)$:
$$
I_2(t)=\frac{\sin(\pi\nu)}{\pi}\int_0^{t\wedge\tau_n}(t\wedge\tau_n-s)^{\nu-1}S(t\wedge\tau_n-s)Z_\nu^n(s)\,ds,
$$
where the auxiliary process $Z_\nu^n$ is defined by
$$
Z_\nu^n(t):=\int_0^t(t-s)^{-\nu}S(t-s)\sigma(u_n(s))\mathbbm{1}_{\{s\le\tau_n\}}\,dW(s),\qquad t\ge0.
$$
Let $\ve$ be a  parameter satisfying $\ve<\min\{1,c_1/2\}=\min\{1,\rho_{qr}/(qr)\}$.  Applying \eqref{eq3.6} (with $\a=0$)  yields
$$\begin{aligned}
&\mathbb{E}\big(\|Z_\nu^n(t)\|_{L^q}^q\big)\\
\le& C_{q,|\mathcal O|}\Theta^\frac{q}{2}\,
\E\Bigg(\int_0^t(t-s)^{-2\nu}e^{-2\lambda(t-s)}\|\sig(u_n(s))\|_{L^q}^2\mathbbm{1}_{\{s\le\tau_n\}}\,ds\Bigg)^{q/2} \\
=& C_{q,|\mathcal O|}\Theta^\frac{q}{2}\,
\E\Bigg(\int_0^t[(t-s)^{-2\nu}e^{-2\lambda(1-\ve)(t-s)}]\.[e^{-2\lam\ve(t-s)}\|\sig(u_n(s))\|_{L^q}^2\mathbbm{1}_{\{s\le\tau_n\}}]\,ds\Bigg)^{q/2} \\
\le& C_{q,|\mathcal O|}\Theta^\frac{q}{2}\,
\(\int_0^t(t-s)^{-2\nu\frac{q}{q-2}}e^{-2\lambda(1-\ve)\frac{q}{q-2}(t-s)}ds\)^{\frac{q-2}{2}}\E\int_0^te^{-q\lam\ve(t-s)}\|\sig(u_n(s))\|_{L^q}^q\mathbbm{1}_{\{s\le\tau_n\}}\,ds\\
\le&C_{q,|\mathcal O|}\Theta^\frac{q}{2}\Gamma_2^\frac{q-2}{2}\E\int_0^te^{-q\lam\ve(t-s)}\|\sig(u_n(s))\|_{L^q}^q\mathbbm{1}_{\{s\le\tau_n\}}\,ds. \end{aligned}$$
By  the growth condition \eqref{eq4.16},
$$\begin{aligned}
&\E\int_0^te^{-q\lam\ve(t-s)}\|\sig(u_n(s))\|_{L^q}^q\mathbbm{1}_{\{s\le\tau_n\}}\,ds\\
\le&\int_0^te^{-q\lam\ve(t-s)}e^{-\rho_{qr}s}\|u_0\|_{L^{qr}}^{qr}\,ds+\int_0^te^{-q\lam\ve(t-s)}ds\\
\le&e^{-q\lam\ve t}\|u_0\| _{L^{qr}}^{qr}\int_0^te^{-(\rho_{qr}-q\lam\ve)s}\,ds+\int_0^te^{-q\lam\ve(t-s)}ds\\
\le&C_{q,r,\ve}(e^{-q\lam\ve t}\|u_0\| _{L^{qr}}^{qr}+1).\end{aligned}$$
Therefore \be\label{eq4.17}\begin{aligned} \mathbb{E}\big(\|Z_\nu^n(t)\|_{L^q}^q\big)\le&C_{q,|\mathcal O|,\nu,\lam,r,\ve}\Theta^\frac{q}{2}(e^{-q\lam\ve t}\|u_0\| _{L^{qr}}^{qr}+1)\\
\le&C_{q,|\mathcal O|,\Theta,\nu,\lam,r,\ve}\big(\|u_0\|_{L^{qr}}^{qr}+1\big).\end{aligned}\ee
It then follows that 
\begin{equation}\label{eq4.6}
\begin{split}
\mathbb{E}\big(\|I_2(t) \|_{C_b}^q\big) \le & \E\(\int_0^{t\wedge\tau_n}(t\wedge\tau_n-s)^{\nu-1}\|S(t\wedge\tau_n-s)Z_\nu^n(s)\|_{C_b}ds\)^q\\
\le&C_{q,d}\E\(\int_0^{t\wedge\tau_n}(t\wedge\tau_n-s)^{-(1+\frac{d}{2q}-\nu)\frac{q}{q-1}}e^{-\frac{\lam}{2}\frac{q}{q-1}(t\wedge\tau_n-s)}ds\)^{q-1}\int_0^t\|Z_\nu^n(s)\|_{L^q}^qds\\
\le&C_{q,d}\Gamma_3^{q-1}\int_0^t\E\big(\|Z_\nu^n(s)\|_{L^q}^q\big)ds\\
\le&C_{q,|\mathcal O|,\Theta,\nu,\lam,d,r,\ve}\big(\|u_0\|_{L^{qr}}^{qr}+1\big)t.
\end{split}
\end{equation}

Combining \eqref{eq4.10} and \eqref{eq4.6}, we deduce that for any $T>0$,
\be\label{eq4.14}\sup_{0\le t\le T}\E\big(\|u_n(t)\|_{C_b}^q\big)\le C(T+1)=:C_1,\ee
where the constant $C_1>0$ depends on the relevant parameters, but is independent of $n$.
\Vs

Now we turn to the H\"older continuity of $u_n$ in time. For any $0\le t_1\le t_2\le T$, we consider the decomposition
\begin{equation*}
\begin{split}
&\E\big(\|J_1\|_{C_b}^q\big)+\E\big(\| J_2\|_{C_b}^q\big)\\
\le&\,C_\gamma |t_2-t_1|^{q\gamma}
\E\Big(\|A_q^{\gamma}S(t_1\wedge\tau_n)u_0\|_{L^q}+\\
&\Hs\Hs\hs+\int_0^{t_1\wedge\tau_n}\|A_q^{\gamma}S[(t_1\wedge\tau_n-s)/2]\|_{\cL(C_b)}\|S[(t_1\wedge\tau_n-s)/2]f(u_n(s))\|_{C_b}\,ds\Big)^q\\
\le&\,C_\gamma |t_2-t_1|^{q\gamma}
\E\Big(\|A_q^{\gamma}u_0\|_{L^q}
+\int_0^{t_1\wedge\tau_n}\|A_q^{\gamma}S[(t_1\wedge\tau_n-s)/2]\|_{\mathcal L(L^q(\mathcal O))}\|f(u_n(s))\|_{L^q}\,ds\Big)^q\\
\le&\,C_{\gamma,q,d}\,|t_2-t_1|^{q\gamma}
\E\Big(\|A_q^{\gamma}u_0\|_{L^q}
+\int_0^{t_1\wedge\tau_n}(t_1\wedge\tau_n-s)^{-\gamma-\frac{d}{2q}}e^{-\lambda\frac{t_1\wedge\tau_n-s}{2}} \big( \| u_n(s) \|_{L^{qr}}^{r} + 1 \big)\,ds\Big)^q\\
\le&\,C_{\gamma,q,d,|\mathcal O|}\,|t_2-t_1|^{q\gamma}
\Big[\|A_q^{\gamma}u_0\|_{L^q}^q+\Gamma_4^{q-1}
         \Big(\E\int_0^t \mathbbm{1}_{\{s \le\tau_n\}} \| u_n(s) \|_{L^{qr}}^{qr} \, ds  +1\Big) \Big]\\
\le&\,C_{\gamma,q,d,|\mathcal O|,\|A^\gam u_0\|_{L^q},\|u_0\|_{L^{qr}},\lambda}\,|t_2-t_1|^{q\gamma}.
\end{split}
\end{equation*}
Similarly, for $J_3$ we have
\begin{equation*}
\begin{split}
\E\big(\|J_3\|_{C_b}^q\big) \le&\, C_{q,d}\E\(\int_{t_1\wedge\tau_n}^{t_2\wedge\tau_n}(t_2\wedge\tau_n-s)^{-\frac{d}{2q}}\|f(u_n(s))\|_{L^q}ds\)^q\\
\le&\,C_{q,d}\E\(\int_{t_1\wedge\tau_n}^{t_2\wedge\tau_n}(t_2-\tau_n-s)^{-\frac{d}{2(q-1)}}ds\)^{q-1}\int_{t_1}^{t_2}\big( \mathbbm{1}_{\{s\le\tau_n\}}\| u_n(s) \|_{L^{qr}}^{qr} + 1 \big)ds\\
\le&\, C_{\gamma,n,|\mathcal O|,\|u_0\|_{L^{qr}},\lambda}|t_2-t_1|^{q-\frac{d}{2}}.
\end{split}
\end{equation*}

For the stochastic terms, we proceed as follows. For $J_4$, using the factorization formula, we get
\begin{equation*}
\begin{split}
&\mathbb{E}\big(\|J_4(t)\|_{C_b}^q\big)\\
 \le & \E\(\int_0^{t_1\wedge\tau_n}(t_1\wedge\tau_n-s)^{\nu-1}\|\big(S(h)-I\big)S(t_1\wedge\tau_n-s)Z_\nu^n(s)\|_{C_b}ds\)^q\\
\le&C_{q,d}\E\[\(\int_0^{t_1\wedge\tau_n}(t_1\wedge\tau_n-s)^{-(1+\frac{d}{2q}-\nu+\gam)\frac{q}{q-1}}e^{-\lam\frac{(t_1\wedge\tau_n-s)q}{2{q-1}}}\)^{q-1}h^{q\gam}\int_0^{t_1}\|Z_\nu^n(s)\|_{L^q}^qds\]\\
\le&C_{q,d}\Gamma_5^{q-1}|t_1-t_2|^{q\gam}\int_0^{t_1}\E\big(\|Z_\nu^n(s)\|_{L^q}^q\big)ds\\
\le&C_{q,|\mathcal O|,\Theta,\nu,\lam,d}\big(\|u_0\|_{L^{qr}}^{qr}+1\big)t_1|t_1-t_2|^{q\gam}.
\end{split}
\end{equation*}
Likewise, for $J_5$, we obtain
\begin{equation*}
\begin{split}
\mathbb{E}\big(\|A_q^\a J_5(t) \|_{L^q}^q\big) \le
 & \E\(\int_{t_1\wedge\tau_n}^{t_2\wedge\tau_n}(t_2\wedge\tau_n-s)^{\nu-1}\|S(t_2\wedge\tau_n-s)Z_\nu^n(s)\|_{C_b}ds\)^q\\
\le&C_{q,d}\E\(\int_{t_1\wedge\tau_n}^{t_2\wedge\tau_n}(t_2\wedge\tau_n-s)^{-(1+\frac{d}{2q}-\nu)\frac{q}{q-1}}ds\)^{q-1}\int_{t_1\wedge\tau_n}^{t_2\wedge\tau_n}\|Z_\nu^n(s)\|_{L^q}^qds\\
\le&C_{q,d}|t_1-t_2|^{q\nu-1-\frac{d}{2}}\E\int_{t_1\wedge\tau_n}^{t_2\wedge\tau_n}\|Z_\nu^n(s)\|_{L^q}^qds\\
\le&C_{q,|\mathcal O|,\Theta,\nu,\lam,d}\big(\|u_0\|_{L^{qr}}^{qr}+1\big)|t_1-t_2|^{q\nu-\frac{d}{2}}.
\end{split}
\end{equation*}
\vs
Finally, collecting all the estimates for $J_1$ through $J_5$, and setting $\xi:=\min\{q\gam,q-d/2,q\nu-d/2\}$, we arrive at the H\"older-type estimate
\be\label{eq4.13}\E\big(\|u_n(t_1)-u_n(t_2)\|_{C_b}^q\big)\le C_2|t_1-t_2|^\xi,\ee
where $C_2>0$ is independent of $n$.

Since $\xi > 1$, applying Corollary \ref{c3} to \eqref{eq4.14} and \eqref{eq4.13}, we get that for any $\xi < \xi' < 2\xi - 1$ there is a constant $C_{q,\xi'} > 0$ such that 
$$
\mathbb{E}\big(\sup_{0\le t\le T}\|u_n(t)\|_{C_b}^q\big)
\le 2^{q-1}(C_1 + C_2 C_{q,\xi'} T^{\xi'})=:C,
$$
where $C = C(q,|\mathcal{O}|,\gamma, d, \lam,\alpha,T,\mathbb{E}\| u_0 \|_{C_b}^q,\mathbb{E}\| u_0 \|_{L^{qr}}^{qr},r)>0$ is  independent of $n$.
\vs
Markov's inequality yields
$$
\mathbb{P}(\tau_n < T) \le \mathbb{P}\!\Big( \sup_{0 \le t \le T} \| u_n(t) \|_{C_b}^q \ge n \Big)
                     \le \frac{\mathbb{E}\!\Big( \sup_{0 \le t \le T} \| u_n(t) \|_{C_b}^q \Big)}{n^q}
                     \le \frac{C}{n^q}.
$$
For each fixed $T > 0$, the Borel--Cantelli lemma implies that almost surely
$\tau_n \ge T$ for all sufficiently large $n$.  
Since $T$ is arbitrary, $\tau_n \to \infty$ almost surely; hence the solution is global.\end{proof}
\vs
\bc Assume \hyperref[H1*]{(H1*)} and \hyperref[H2]{(H2)} hold. Then for any $u_0 \in D(A_0^{1/2})$, the mild solution $u$ of \eqref{eq} satisfies\begin{equation*}\label{eq4.15}
\mathbb{E}\big(\|u(t)\|_{L^{qr}}^{qr}\big)
\le e^{-\rho_qt}\|u_0\|_{L^{qr}}^{qr}
+R_q,\Hs t\ge0,
\end{equation*} where $\rho_{qr}:=c_1 qr/2$.\ec
\begin{proof} Since $\tau_n\to\infty$ as $n\to\infty$, the desired result follows by passing to the limit $n\to\infty$ in \eqref{eq:Lq-bound}.\end{proof}

Having established the necessary estimates, we are now in a position to prove Theorem \ref{th:global}.
\begin{proof}[Proof of Theorem \ref{th:global}]  For $t \ge 1$, we decompose the mild solution as
\begin{equation*}
\begin{split}
u(t) =&\, S(1) u(t-1) + \int_{t-1}^t S(t-s) f(u(s)) \, ds \\
      &+ \int_{t-1}^t S(t-s) \sigma(u(s)) \, dW(s) \\
=:&\,S(1) u(t-1)+ K_1(t) + K_2(t).
\end{split}
\end{equation*}
We first know from \eqref{eq3.4}
 that $\| S(1) u(t-1) \|_{C_b} \le C_{q,d} \| u(t-1) \|_{L^q}$. By Lemma \ref{le4.2}, 
\begin{equation*}
\begin{split}
\mathbb{E}\big(\| u(t-1) \|_{L^q}^q\big)\le C_{q,d}(e^{-\rho_q(t-1)}+1), \Hs t \ge 1.
\end{split}\end{equation*}
Then \be\label{eq4.3}\E\big(\|S(1) u(t-1) \|_{C_b}^q\big) \le C_{q,d} \mathbb{E}\big(\|u(t-1)\|_{L^q}^q\big)\le C_{q,d}\big(e^{-\rho_q(t-1)}\| u_0 \|_{L^q}^q + 1 \big), \Hs t \ge 1.\ee
\vs
For the deterministic part $K_1$, similar to \eqref{eq4.2},
\begin{equation}\label{eq4.12}\begin{split}
      \int_{t-1}^{t } \big\|S(t  - s) f(u(s)) \big\|_{C_b} \, ds 
\le&\, C_{q,d}\Big( \Gamma_1^{\frac{q-1}{q}}
        \Big( \int_{t-1}^t  \| u(s) \|_{L^{qr}}^{qr} \, ds \Big)^{\frac{1}{q}} + \Gamma_2 \Big)\\
\le&\, C_{q,d,\lambda} \Big[ \Big( \int_{t-1}^t  \| u(s) \|_{L^{qr}}^{qr} \, ds \Big)^{\frac{1}{q}} + 1 \Big],\end{split}
\end{equation}
where $C_{q,d,\lambda}$ depends on $q$, $d$, $\lambda$ (via $\Gamma_1$, $\Gamma_2$ and $C_{q,d}$).
From Lemma \ref{le4.2} we have 
$$
\int_{t-1}^t \mathbb{E}\big(\| u(s) \|_{L^{qr}}^{qr} \big)\, ds
\le \,e^{-\rho_{qr}(t-1)} \| u_0 \|_{L^{qr}}^{qr} + R_{qr}, \Hs t \ge 1.
$$
Inserting it into \eqref{eq4.12} yields 
\begin{equation}\label{eq3.18}
\begin{split}
\mathbb{E}\big(\| K_1(t)\|_{C_b}^q\big)\le&\,C_{q,d}\E\Big[\Big(\int_{t-1}^t \|S(t-s)f(u(s))\|_{C_b}ds\Big)^q\Big]\\
\le&\, C_{q,d,\lambda, r} \big(e^{-\rho_{qr}(t-1)}\| u_0 \|_{L^{qr}}^{qr} + 1 \big).
\end{split}
\end{equation}

Recall the constants $\Gamma_4$ and $\Gamma_5$ depend on $q$, $d$, $\alpha$, $\lambda$. For the stochastic term $K_2(t)$, 
we have$$\begin{aligned}
K_2(t)=&\frac{\sin(\pi\nu)}{\pi}\int_0^{t}(t-s)^{\nu-1}S(t-s)Z_\nu(s)\,ds\\
&-\frac{\sin(\pi\nu)}{\pi}\int_0^{t-1}(t-1-s)^{\nu-1}S(t-1-s)Z_\nu(s)\,ds\\
=:&L_1(t)+L_2(t).\end{aligned}
$$
\eqref{eq4.17} implies that $$\mathbb{E}\big(\|Z_\nu(t)\|_{L^q}^q\big)\le C_{q,|\mathcal O|,\nu,\lam,r,\ve}\Theta^\frac{q}{2}(e^{-q\lam\ve t}\|u_0\| _{L^{qr}}^{qr}+1).$$
Suppose that $0<\epsilon<\ve$. Then
\begin{equation*}
\begin{split}
&\mathbb{E}\big(\|L_1(t) \|_{C_b}^q\big)\\
 \le & \E\Big[\Big(\int_0^t(t-s)^{\nu-1}\|S(t-s)Z_\nu(s)\|_{C_b}ds\Big)^q\Big]\\
\le&C_{q,d}\E\Big[\Big(\int_0^{t}(t-s)^{-(1+\frac{d}{2q}-\nu)}e^{-\lam\frac{(t-s)}{2}}\|Z_\nu(s)\|_{L^q}ds\Big)^q\Big]\\
=&C_{q,d}\E\Big[\Big(\int_0^{t}[(t-s)^{-(1+\frac{d}{2q}-\nu)}e^{-\lam(1-\epsilon)\frac{(t-s)}{2}}]\.[e^{-\lam\epsilon\frac{(t-s)}{2}}\|Z_\nu(s)\|_{L^q}]ds\Big)^q\Big]\\
\le&\,C_{q,d}\(\int_0^{t}\((t-s)^{-(1+\frac{d}{2q}-\nu)}e^{-\lam(1-\epsilon)\frac{(t-s)}{2}}\)^\frac{q}{q-1}ds\)^{q-1}\int_0^te^{-q\lam\epsilon\frac{(t-s)}{2}}(e^{-q\lam\ve s}\|u_0\| _{L^{qr}}^{qr}+1)ds\\
\le& C_{q,|\mathcal O|,\nu,\lam,r,\ve,\epsilon}\Theta^\frac{q}{2}\big(e^{-\frac{q\lam\epsilon t}{2}}\| u_0 \|_{L^{qr}}^{qr} + 1 \big).
\end{split}
\end{equation*}
Similarly, $$\mathbb{E}\big(\|L_2(t) \|_{C_b}^q\big) \le C_{q,|\mathcal O|,\nu,\a,r,\ve,\epsilon}\Theta^\frac{q}{2}\big( e^{-\frac{ qr\epsilon}{2}(t-1)}\| u_0 \|_{L^{qr}}^{qr} + 1 \big).$$
As a result, 
\be\label{eq4.18}\mathbb{E}\big(\|K_2(t) \|_{C_b}^q\big) \le C_{q,|\mathcal O|,\nu,\a,r,\ve,\epsilon}\Theta^\frac{q}{2}\big( e^{-\frac{qr\epsilon}{2}(t-1)}\| u_0 \|_{L^{qr}}^{qr} + 1 \big).\ee 
\vs
Combining \eqref{eq4.3}, \eqref{eq3.18} and \eqref{eq4.18},  we conclude that
\be\label{eq4.4}
\mathbb{E}\big(\| u(t) \|_{C_b}^q\big)\le C (\Theta^\frac{q}{2}+1)\big( e^{-\frac{qr\epsilon}{2}(t-1)}\| u_0 \|_{L^{qr}}^{qr} + 1 \big), \Hs t \ge 1,
\ee
where $C = C(q,|\mathcal{O}|,d,\alpha,r)> 0$. It is worth noting that, thanks to the choice of $\ve$, the parameter $\epsilon>0$ can be chosen arbitrarily, subject to
$$
0<\epsilon<\min\left\{1,\frac{c_1}{2}\right\}.
$$
The proof is complete.
\end{proof}
\subsection{Proof of Theorem \ref{th}}
For each $m\in\mathbb{N}^+$, define $$\mu_{jm}:=\left\{\ba{lll}\hs\mu_j,\Hs\hs \lam_j\le1;\\[1ex]\lam_j^{-1/m}\mu_j,\,\hs\hs\lam_j>1 \ea\right.\quad\hbox{ and }\quad W_m(t):=\sum_{j=1}^{\infty} \sqrt{\mu_{jm}} e_j B_j(t),\qquad j\in\mathbb{N}^+.$$
Then for each $j\in\mathbb{N}^+$, $\mu_{jm}\uparrow \mu_j$ as $m\to\infty$. 
Since $\lambda_j \to \infty$ as $j\to\infty$, there are only finitely many terms with $\lambda_j \le 1$. 
Consequently, $W_m(t)$ satisfies \hyperref[H1*]{(H1*)} with $\delta = 1/m$, 
and$$
\Theta_m := \sum_{j=1}^\infty \mu_{jm} \|e_j\|_{C_0}^2 \uparrow \Theta \Hs \text{as } m\to\infty.
$$
Moreover,
$$
\Theta_{mn}:=\sum_{j=1}^{\infty} |\mu_{jm}-\mu_{jn}| \| e_j \|_{C_b}^2 \to 0 \Hs \text{as } m,n\to\infty.
$$
Indeed, for each fixed $j$, $|\mu_{jm}-\mu_{jn}| \to 0$ as $m,n\to\infty$. 
Since $0 \le \mu_{jm} \le \mu_j$ for all $m$, we have $|\mu_{jm}-\mu_{jn}| \le 2\mu_j$. 
The assumption $\Theta = \sum_{j=1}^\infty \mu_j \|e_j\|_{C_b}^2 < \infty$ (cf. \hyperref[H1]{(H1)}) then allows us to apply the dominated convergence theorem to the series, yielding $\Theta_{mn} \to 0$.
\vs
Since $\Theta_m \le \Theta$, the functions $f$ and $\sigma$ retain the dissipativity and one-sided Lipschitz properties with $\Theta$ replaced by $\Theta_m$. Specifically, from \hyperref[H2]{(H2)} and \hyperref[H3]{(H3)} we obtain
\begin{equation}\label{eq4.9}
f(u)u + (q-1)\Theta_m |\sigma(u)|^2 \le -c_1 u^2 + M_{q,r},\qquad u\in\mathbb{R},
\end{equation}
and
\begin{equation}\label{eq4.22}
(f(u)-f(v))(u-v) + (q-1)\Theta_m|\sigma(u)-\sigma(v)|^2 \le N_{q,r} |u-v|^2 ,\qquad u,v \in \mathbb{R}.
\end{equation}

In view of \eqref{eq4.9}, Theorem \ref{th:global} directly implies the following result.
\begin{lemma}\label{le4.6}
Let $u^m(t)$, $t\ge0$, be the solution of \eqref{eq} with initial data  
$u_0^m \in D(A_0^{1/2})$ driven by the Wiener process $W(t):=W_m(t)$. Then for any $q\ge2$, we have
\begin{equation}\label{eq3.17'}
\mathbb{E}\big(\| u^m(t) \|_{L^q}^q\big) \le \| u_0^m \|_{L^q}^q e^{-\rho_q t} 
                                + C_q, \Hs t \ge 0,
\end{equation}
and
\begin{equation}\label{eq4.11}
\mathbb{E}\big(\| u^m(t) \|_{C_b}^q\big) \le C\big(\| u_0^m \|_{L^{qr}}^{qr} e^{-\varrho_{qr}(t-1)} + 1 \big), \Hs t \ge 1,
\end{equation}
where $\rho_q:=c_1q/2$, $\varrho_{qr}:=qr\min\{1/2,c_1/4\}$ and $C = C(q,|\mathcal{O}|,d,\alpha,r,\Theta) > 0$ is independent of $m$.
\end{lemma}

\begin{proof}
Thanks to \eqref{eq4.9}, estimate \eqref{eq3.17'} follows by the same argument as in \eqref{eq:Lq-bound}. 
\vs
Invoking \eqref{eq4.4} and the fact $\O_m\le\O$, we obtain
\begin{align*}
 \mathbb{E}\big(\| u^m(t) \|_{C_b}^q\big) 
&\le C (\Theta_m^{q/2} + 1) \big(\| u_0^m \|_{L^{qr}}^{qr} e^{-\varrho_{qr}(t-1)} + 1 \big) \\
&\le C (\Theta^{q/2} + 1) \big( \| u_0^m \|_{L^{qr}}^{qr} e^{-\varrho_{qr}(t-1)} + 1 \big), \Hs t \ge 1,
\end{align*}
which yields \eqref{eq4.11}.
\end{proof}

\begin{proof}[Proof of Theorem \ref{th}] We first assume $q>d+6$.

{\bf Step 1. Cauchy sequence in $L^q(\Omega;C_b(\mathcal{O}))$.} Since $D(A_b^{1/2})$ is dense in $C_0(\overline{\mathcal O})$, we can choose a sequence $\{u_0^m\}_{m\in\mathbb N^+}\subset D(A_b^{1/2})$ converging to $u_0$ in $C_0(\overline{\mathcal O})$. The goal is to show that for each $t\ge0$, $\{u^m(t)\}$ is Cauchy in $L^q(\Omega; C_b(\mathcal{O}))$.

 Consider two solutions $u^m(t), u^n(t)$, $t\ge0$, with initial data  
$u_0^m, u_0^n \in D(A_b^{1/2})$, driven by the Wiener process $W_m(t)$ and $W_n(t)$, respectively. 
Set $u^{mn}(t) := u^m(t) - u^n(t)$ and $u_0^{mn} := u_0^m-u_0^n$.  
Then $u^{mn}(t)$ satisfies  
\begin{equation*}
\begin{aligned}
u^{mn}(t)=&\, S(t) u_0^{mn} 
        + \int_0^t S(t-s) \big[ f(u^m(s)) - f(u^n(s)) \big] \, ds \\
     &\,+\int_0^t  S(t-s) \sum_{j=1}^\infty\big[ \sigma(u^m(s))\sqrt{\mu_{jm}} - \sigma(u^n(s))\sqrt{\mu_{jn}} \big]e_j \, dB_j(s)\\
     =:&\, S(t) u_0^{mn} +P_1(t)+P_2(t),\Hs t\ge0.
\end{aligned}
\end{equation*}

Applying the It\^o formula to $\|u^{mn}(t)\|_{L^{qr}}^{qr}$ introduces two additional terms due to the different noise intensities $\mu_{jm}$ and $\mu_{jn}$:
\begin{itemize}
    \item $\Theta_{mn}:=\sum_j |\mu_{jm}-\mu_{jn}|\,\|e_j\|_{C_b}^2$ (measuring the noise approximation error);
    \item $F(t):=\mathbb{E}\int_0^t\int_{\mathcal{O}} |u^{mn}(s)|^{qr-2}|\sigma(u^n(s))|^2 dxds$, which is controlled using the moment bounds for $u^n$ from Lemma \ref{le4.6}.
\end{itemize}
The first estimate is obtained via Gronwall's inequality, yielding \eqref{eq4.34}; the second follows directly from the preceding estimates, giving \eqref{eq4.7}. Convergence follows since $\Theta_{mn}\to0$ and $u_0^m\to u_0$ in $C_0(\overline{\mathcal{O}})$.

We now carry out the detailed estimates. For simplicity, denote $\vartheta:=qr$. Applying It\^o formula to $\|u^{mn}(t)\|_{L^\vartheta}^\vartheta$ and taking expectation, we obtain  
\begin{equation}\label{eq4.29}
\begin{split}
&\,\mathbb{E}\big(\|u^{mn}(t)\|_{L^\vartheta}^\vartheta\big) +\vartheta(\vartheta-1) \mathbb{E}\int_0^t \int_{\mathcal{O}} |u^{mn}(s)|^{\vartheta-2} \sum_{i,j=1}^d a_{ij}(x) \frac{\partial u^{mn}(s)}{\partial x_i} \frac{\partial u^{mn}(s)}{\partial x_j} \, dx ds\\
= &\, \|u^{mn}_0\|_{L^\vartheta}^\vartheta
+ \vartheta \mathbb{E}\int_0^t \int_{\mathcal{O}} |u^{mn}(s)|^{\vartheta-2} u^{mn}(s) \big( f(u^m(s)) - f(u^n(s)) \big) \, dx \, ds \\
&\, + \frac12 {\vartheta}({\vartheta}-1) \mathbb{E}\int_0^t\sum_{j=1}^\infty \int_{\mathcal{O}} |u^{mn}(s)|^{\vartheta-2} |G_j(s)|^2 \, dx \, ds,
\end{split}
\end{equation}
where 
$$
G_j(s) := \big[ \sigma(u^m(s))\sqrt{\mu_{jm}} - \sigma(u^n(s))\sqrt{\mu_{jn}} \big] e_j.
$$

Since 
$$
|G_j(s)|^2\le 2\big[|\sigma(u^m(s))- \sigma(u^n(s))|^2\mu_{jm}+|\sigma(u^n(s))|^2|\mu_{jm}-\mu_{jn}|\big]\|e_j\|_{C_b}^2,
$$
we have
\begin{equation}\label{eq4.30}
\begin{split}
&\,\sum_{j=1}^\infty \int_{\mathcal{O}} |u^{mn}(s)|^{\vartheta-2} |G_j(s)|^2 \, dx \\
\le&\,2\sum_{j=1}^\infty  \int_{\mathcal{O}} |u^{mn}(s)|^{\vartheta-2} |\sigma(u^m(s))- \sigma(u^n(s))|^2\mu_{jm}\|e_j\|_{C_b}^2 \, dx\\
&\,+2\sum_{j=1}^\infty  \int_{\mathcal{O}} |u^{mn}(s)|^{\vartheta-2} |\sigma(u^n(s))|^2|\mu_{jm}-\mu_{jn}|\|e_j\|_{C_b}^2  dx\\
\le&\,2\Theta_m  \int_{\mathcal{O}} |u^{mn}(s)|^{\vartheta-2} |\sigma(u^m(s))- \sigma(u^n(s))|^2dx  \\
&\,+2\Theta_{mn}\int_{\mathcal{O}} |u^{mn}(s)|^{\vartheta-2} |\sigma(u^n(s))|^2 dx,
\end{split}
\end{equation}
where $\Theta_m=\sum_{j=1}^\infty\mu_{jm}\|e_j\|_{C_b}^2$ and $\Theta_{mn}=\sum_{j=1}^\infty|\mu_{jm}-\mu_{jn}|\.\|e_j\|_{C_b}^2$.
\vs
Inserting \eqref{eq4.30} into \eqref{eq4.29} and applying \eqref{eq4.22}, we obtain
\begin{equation}\label{eq4.31}
\begin{split}
&\,\mathbb{E}\big(\|u^{mn}(t)\|_{L^\vartheta}^\vartheta\big) +\vartheta(\vartheta-1) \mathbb{E}\int_0^t \int_{\mathcal{O}} |u^{mn}(s)|^{\vartheta-2} |\nabla u^{mn}(s)|^2 \, dx \, ds\\
\le &\,\|u^{mn}_0\|_{L^\vartheta}^\vartheta
+ \vartheta \mathbb{E}\int_0^t \int_{\mathcal{O}} |u^{mn}(s)|^{\vartheta-2}
        \big[  u^{mn}(s) \big( f(u^m(s)) - f(u^n(s)) \big)+ \\
&\Hs\Hs\Hs\Hs\Hs\Hs + \Theta_m (\vartheta-1) | \sigma(u^m(s)) - \sigma(u^n(s)) |^2 \big] dx \, ds \\
&\, + \vartheta(\vartheta-1)\Theta_{mn} \mathbb{E}\int_0^t\int_{\mathcal{O}} |u^{mn}(s)|^{\vartheta-2} |\sigma(u^n(s))|^2 dx\, ds\\
\le&\,\|u^{mn}_0\|_{L^\vartheta}^\vartheta
+ \vartheta N_{\vartheta,r}\mathbb{E}\int_0^t \|u^{mn}(s)\|_{L^\vartheta}^{\vartheta}ds +\vartheta(\vartheta-1)\Theta_{mn} F(t),
\end{split}
\end{equation}
where 
\begin{align*}
F(t):=&\,\mathbb{E}\int_0^t\int_{\mathcal{O}} |u^{mn}(s)|^{\vartheta-2} |\sigma(u^n(s))|^2 dx\, ds\\
\le&\,\mathbb{E}\Big(\int_0^t\|u^{mn}(s)\|_{L^\vartheta}^\vartheta ds+\int_0^t\|\sigma(u^n(s))\|_{L^\vartheta}^\vartheta ds\Big)\\
\le&\,\mathbb{E}\int_0^t\|u^{mn}(s)\|_{L^\vartheta}^\vartheta ds+c_2^\vartheta \, 2^{\vartheta-1} \mathbb{E}\int_0^t \big( \|u^n(s)\|_{L^{\vartheta r}}^{\vartheta r} + |\mathcal{O}| \big)ds\\
\stackrel{\eqref{eq3.17'}}{\le}&\,\mathbb{E}\int_0^t\|u^{mn}(s)\|_{L^\vartheta}^\vartheta ds + C_{q,r,|\mathcal{O}|}\big( \|u_0^n\|_{L^{\vartheta r}}^{\vartheta r} +t\big).
\end{align*}

Inserting the estimate of $F(t)$ into \eqref{eq4.31} and noting $\Theta_{mn}\le 2\Theta$, we obtain 
\begin{equation}\label{eq4.37}
\begin{split}
\mathbb{E}\big(\|u^{mn}(t)\|_{L^\vartheta}^\vartheta\big)\le&\,\|u^{mn}_0\|_{L^\vartheta}^\vartheta +C_{q,r,|\mathcal{O}|}\vartheta(\vartheta-1)\big(\|u_0^n\|_{L^{\vartheta r}}^{\vartheta r} + t \big)\Theta_{mn}
\\
&\,+ \vartheta \big[ N_{\vartheta,r}+(\vartheta-1)\cdot 2\Theta \big] \,\mathbb{E}\int_0^t \|u^{mn}(s)\|_{L^\vartheta}^{\vartheta}ds\\
=:& \,\|u^{mn}_0\|_{L^\vartheta}^\vartheta+C_n(t)\,\Theta_{mn}+C'\mathbb{E}\int_0^t \|u^{mn}(s)\|_{L^\vartheta}^{\vartheta}ds,
\end{split}
\end{equation}
where 
$$
C_n(t):=C_{q,r,|\mathcal{O}|}\vartheta(\vartheta-1)\big(\|u_0^n\|_{L^{\vartheta r}}^{\vartheta r} + t \big)\quad\hbox{and} \quad
C':=\vartheta\big[ N_{\vartheta,r}+(\vartheta-1)\cdot 2\Theta \big].
$$
Applying Gronwall's inequality to \eqref{eq4.37} yields
\begin{equation}\label{eq4.34}
\begin{aligned}
\mathbb{E}\big(\|u^{mn}(t)\|_{L^{qr}}^{qr}\big)
\le&\ \|u^{mn}_0\|_{L^{qr}}^{qr} + C_n(t)\Theta_{mn} \\
&+ C' \int_0^t \big[\|u^{mn}_0\|_{L^{qr}}^{qr} + C_n(s)\Theta_{mn} \big] e^{C'(t-s)} \, ds \\
\le&e^{C't}\big(\|u^{mn}_0\|_{L^{qr}}^{qr} + C_n(t)\Theta_{mn} \big).
\end{aligned}
\end{equation}
This estimate will be useful for estimating $\mathbb{E}\|u^{mn}(t)\|_{C_b}^q$ in the following.
\vs
Using \eqref{eq3.4},
\begin{equation}\label{eq4.35}
\begin{aligned}
\mathbb{E}\big(\|P_1(t)\|_{C_b}^q\big)
&\le \mathbb{E}\Big[\Big( \int_0^t \|S(t-s) [ f(u^m(s)) - f(u^n(s)) ]\|_{C_b} \, ds \Big)^q\Big] \\
&\le C_q\mathbb{E}\Big[\Big( \int_0^t (t-s)^{-d/{2q}}e^{-\lambda(t-s)}\|f(u^m(s)) - f(u^n(s)) \|_{L^q} \, ds \Big)^q\Big] \\
&\le C_q \Gamma_2^{q-1}\, \mathbb{E}\int_0^t \| f(u^m(s)) - f(u^n(s)) \|_{L^q}^q ds,
\end{aligned}
\end{equation}
where $\Gamma_2$ is the same constant as that in Theorem \ref{th:global}, depending only on $q,d,\lambda$.
\vs
As for the stochastic term $P_2$, we write 
$$P_2(t)=\frac{\sin(\pi\nu)}{\pi}\int_0^{t}(t-s)^{\nu-1}S(t-s)\~Z_\nu(s)\,ds,$$
where 
$$\begin{aligned}\~Z_\nu(s):=&\int_0^t(t-s)^{-\nu}S(t-s)\sum_{j=1}^\infty\big[ \sigma(u^m(s))\sqrt{\mu_{jm}} - \sigma(u^n(s))\sqrt{\mu_{jn}} \big]e_j\, dB_j(s)\\
=&\int_0^t(t-s)^{-\nu}S(t-s)\big[ \sigma(u^m(s)) - \sigma(u^n(s)) \big]dW_m(s)\\
&+\int_0^t(t-s)^{-\nu}S(t-s)\sum_{j=1}^\infty\sigma(u^n(s))\big(\sqrt{\mu_{jm}} -\sqrt{\mu_{jn}}\big)e_j\, dB_j(s)\\
=:&\,\~Z_\nu^1(s)+\~Z_\nu^2(s).\end{aligned}$$
By the same arguments as those in \eqref{eq3.6}, we have 
\begin{equation*}
\begin{split}
&\,\mathbb{E}\big(\|\~Z_\nu^1(t)\|_{L^q}^q\big)\\
\le&\, C_{q,|\mathcal{O}|,\alpha} \, \mathbb{E}
   \Big[\Big( \sum_{j=1}^{\infty} \mu_{jm} \| e_j \|_{C_b}^2
          \int_0^{t} e^{-\lambda (t - s)} (t - s)^{-2\nu}
          \| \sigma(u^m(s)) - \sigma(u^n(s)) \|_{L^q}^2 ds\Big)^{q/2} \Big] \\
\le&\, C_{q,|\mathcal{O}|,\alpha,\Theta}\,
        \Gamma_2^{\frac{q-2}{2}}\,\E\int_0^t \| \sigma(u^m(s)) - \sigma(u^n(s)) \|_{L^q}^q \, ds\\
        =:&\,C_{q,|\mathcal{O}|,\alpha,\Theta,\nu,\lam}\E\int_0^t \| \sigma(u^m(s)) - \sigma(u^n(s)) \|_{L^q}^q \, ds,
\end{split}
\end{equation*}
and \begin{equation*}
\begin{split}
\mathbb{E}\big(\|\~Z_\nu^2(t)\|_{L^q}^q\big)
\le&\, C_{q,|\mathcal{O}|,\alpha} \, \Theta_{mn}\mathbb{E}
          \Big[\Big(\int_0^{t} e^{-\lambda (t - s)} (t - s)^{-2\nu}
          \|\sigma(u^n(s)) \|_{L^q}^2 ds\Big)^{q/2}\Big] \\
\le&\, C_{q,|\mathcal{O}|,\alpha}\Theta_{mn}\,
        \Gamma_2^{\frac{q-2}{2}}\,\int_0^t \E\big(\|\sigma(u^n(s)) \|_{L^q}^q\big) \, ds\\
\le &\,C_{q,|\mathcal{O}|,\alpha,\nu,\lam}\Theta_{mn}\,
\int_0^t\E\big(\|u^n(s) \|_{L^{qr}}^{qr}\big)+1 \, ds\\
\stackrel{\eqref{eq3.17'}}{\le}&C_{q,|\mathcal{O}|,\alpha,\nu,\lam,d,r,\Theta}\Theta_{mn}\,\big(\|u_0^n\|_{L^{qr}}^{qr}+t \big).
\end{split}
\end{equation*}
It follows that \be\label{eq3.19}
\begin{aligned}
\E\|P_2(t)\|_{C_b}^q
\le& \E\Big[\Big(\int_0^{t}(t-s)^{\nu-1}
\|S(t-s)\~Z_\nu(s)\|_{L^q}\,ds\Big)^q\Big] \\
\le&C_{q,d}\E\[\(\int_0^{t\wedge\tau_n}(t\wedge\tau_n-s)^{-(1+\frac{d}{2q}-\nu)\frac{q}{q-1}}e^{-\frac{\lam}{2}\frac{q}{q-1}(t\wedge\tau_n-s)}ds\)^{q-1}\int_0^t\|\~Z_\nu(s)\|_{L^q}^qds\]\\
\le&C_{q,d}\Gamma_3^{q-1}\int_0^t\E\big(\|\~Z_\nu(s)\|_{L^q}^q\big)ds\\
\le& C_{q,|\mathcal{O}|,\alpha,\Theta,\nu,\lam,d,\lam}\Big(\E\int_0^t \| \sigma(u^m(s)) - \sigma(u^n(s)) \|_{L^q}^q \, ds+\Theta_{mn}\,
        \big(\|u_0^n\|_{L^{qr}}^{qr}+t \big)\Big),
\end{aligned}
\ee
where $0<\epsilon<1$ is any given number.

By the condition \eqref{eq3.5} in \hyperref[H3]{(H3)} and H\"older's inequality, we first have the estimates
\begin{equation*}
\begin{split}
&\,\| f(u^m(t)) - f(u^n(t)) \|_{L^q}\vee\| \sigma(u^m(t)) - \sigma(u^n(t)) \|_{L^q}\\
\le&\, c_3\Big( \int_\mathcal{O} \big(1+|u^m(t)|^{r-1}+|u^n(t)|^{r-1}\big)^q |u^{mn}(t)|^q \,dx \Big)^{1/q} \\
\le&\,  C_{q,r,|\mathcal{O}|}\big( 1+\|u^m(t)\|_{L^{qr}}^{r-1}+\|u^n(t)\|_{L^{qr}}^{r-1} \big) \cdot \|u^{mn}(t)\|_{L^{qr}},\Hs t\ge0,
\end{split}
\end{equation*}
and using H\"older's inequality again,
\begin{equation}\label{eq4.18'}
\begin{split}
&\, \mathbb{E}\int_0^t \| f(u^m(s)) - f(u^n(s)) \|_{L^q}^q \, ds
   \vee \mathbb{E}\int_0^t \| \sigma(u^m(s)) - \sigma(u^n(s)) \|_{L^q}^q \, ds \\
\le&\, C_{q,r,|\mathcal{O}|}\mathbb{E}\int_0^t \big( 1 + \|u^m(s)\|_{L^{qr}}^{r-1} + \|u^n(s)\|_{L^{qr}}^{r-1} \big)^{q} \,
   \|u^{mn}(s)\|_{L^{qr}}^q \, ds \\
\le&\, C_{q,r,|\mathcal{O}|}\left( \mathbb{E}\int_0^t \big( 1 + \|u^m(s)\|_{L^{qr}}^{qr} + \|u^n(s)\|_{L^{qr}}^{qr} \big) \, ds \right)^{(r-1)/r}
      \left( \mathbb{E}\int_0^t \|u^{mn}(s)\|_{L^{qr}}^{qr} \, ds \right)^{1/r}\\
\stackrel{\eqref{eq3.17'}}{\le}&\, C_{mn}(t)\bigg( \mathbb{E}\int_0^t \|u^{mn}(s)\|_{L^{qr}}^{qr} \, ds \bigg)^{1/r},
\end{split}
\end{equation}
where 
$$
C_{mn}(t):= C_{q,r,|\mathcal{O}|} \left[ t+ \|u_0^m\|_{L^{qr}}^{qr}+\|u_0^n\|_{L^{qr}}^{qr} \right]^{(r-1)/r}>0.
$$
Inserting \eqref{eq4.18'} into \eqref{eq4.35} and \eqref{eq3.19}, respectively,  we obtain 
\begin{equation}\label{eq4.7}
\begin{split}
\mathbb{E} \big(\| u^{mn}(t) \|_{C_b}^q\big) 
\le&\, 3^{q-1} \|u_0^{mn}\|_{C_b}^q
   + C\Big[C_{mn}(t)\bigg( \mathbb{E}\int_0^t \|u^{mn}(s)\|_{L^{qr}}^{qr} \, ds \bigg)^{1/r}+\Theta_{mn}\,
        \big(\|u_0^n\|_{L^{qr}}^{qr}+t \big)\Big]\\
\stackrel{\eqref{eq4.34}}{\le}&\, 3^{q-1} \|u_0^{mn}\|_{C_b}^q+ C \Big[C_{mn}(t)\big(t e^{C't}\big)^{1/r}\big(  \|u^{mn}_0\|_{L^{qr}}^{qr}+ C_n(t)\Theta_{mn} \big)^{1/r} +\\
&\Hs\Hs\Hs\hs+\Theta_{mn}\,
        \big(\|u_0^n\|_{L^{qr}}^{qr}+t \big)\Big],\Hs t\ge0,
\end{split}
\end{equation}
where $C = C(q,d,\lambda,|\mathcal{O}|,\alpha,\Theta) > 0$. 
Note that $\lim_{n\to\infty}C_n(t)$ and $\lim_{m,n\to\infty}C_{mn}(t)$ exist, since $u_0^n \to u_0$ in $C_0(\overline{\mathcal{O}})$. 

Recall that $$\lim_{m,n\to\infty}\Theta_{mn}=\lim_{m,n\to\infty}\|u_0^{mn}\|_{L^{qr}}^{qr}=0,$$ and that $$\lim_{n\to\infty}C_n(t)\quad\hbox{and}\quad\lim_{m,n\to\infty}C_{mn}(t)\hs\hbox{exist}.$$ These together with \eqref{eq4.34} and \eqref{eq4.7} implies that $u_m(t)$ is a Cauchy sequence in $L^{qr}(\Omega; L^{qr}(\mathcal{O}))\cap L^q(\Omega; C_0(\overline{\mathcal{O}}))$.
\Vs

{\bf Step 2. Convergence to a solution with original data and noise.} 
Since for any $t\ge0$, $\{u^m(t)\}$ is a Cauchy sequence in $L^{qr}(\Omega; L^{qr}(\mathcal{O}))\cap L^q(\Omega; C_0(\overline{\mathcal{O}}))$. Consequently, there exists a limit process $u(t) \in L^{qr}(\Omega; L^{qr}(\mathcal{O})) \cap L^q(\Omega; C_0(\overline{\mathcal{O}}))$ such that $u^m(t) \to u(t)$.

Passing to the limit $m \to \infty$ and recalling that $\mu_{jm} \to \mu_j$ as $m \to \infty$, we conclude that $u(t)$ is the unique mild solution corresponding to the initial data $u_0\in C_0(\overline{\cO})$ and the Wiener process $W(t)$ satisfying \hyperref[H1]{(H1)}. This establishes the global existence in time of the solution.
\vs
On the other hand, thanks to Lemma \ref{le4.6}, for each $m$, we have
$$
\mathbb{E}\big(\| u^m(t) \|_{C_b}^q\big) \le C \big( \| u_0^m \|_{L^{qr}}^{qr} e^{-\varrho_{qr}(t-1)} + 1 \big), \Hs t \ge 1,
$$
where $\varrho_{qr} := qr\min\{1/2, c_1/4\} > 0$ and $C = C(q,d,|\mathcal{O}|,\Theta,\alpha,r) > 0$. Since both constants are independent of $m$, letting $m \to \infty$, we obtain
$$
\mathbb{E}\big(\| u(t) \|_{C_b}^q\big) \le C \big( \| u_0 \|_{L^{qr}}^{qr} e^{-\varrho_{qr}(t-1)} + 1 \big), \Hs t \ge 1,
$$
where the constant $C>0$  depends only on $q$, $d$, $|\mathcal{O}|$, $\Theta$ and $r$. 
This completes the proof of Theorem \ref{th} in the case of $q>d+6$.
\vs 
Now Suppose that $q\le d+6$. Since  $$\mathbb{E}\big(\| u(t) \|_{C_b}^{d+7}\big) \le C \big(\| u_0 \|_{C_b}^{(d+7)r} e^{-\varrho_{(d+7)r}(t-1)} + 1 \big), \Hs t \ge 1,$$
it follows that
$$\begin{aligned}\mathbb{E}\big(\| u(t) \|_{C_b}^q\big) \le& \[\mathbb{E}\big(\| u(t) \|_{C_b}^{d+7}\big)\]^{\frac{q}{d+7}}\\
\le&C\big( \| u_0 \|_{C_b}^{(d+7)r} e^{-\varrho_{(d+7)r}(t-1)} + 1 \big)^{\frac{q}{d+7}}\\
\le&C\big( \| u_0 \|_{C_b}^{qr} e^{-\varrho_{qr}(t-1)} + 1 \big), \Hs t \ge 1.\end{aligned}$$
The proof is complete.
\end{proof}

\appendix\section*{Appendices}
\addcontentsline{toc}{section}{Appendices} 
\refstepcounter{section}

\subsection{Sectorial operators and analytic semigroups}\label{App2}
Here we recall several definitions and foundational results concerning sectorial operators and analytic semigroups, following the treatment in A. Lunardi \cite{Lu}.

Let $X$ be a Banach space with norm $\|\.\|$, and let $A:D(A)\subset X\ra X$ be a closed linear operator, where the domain $D(A)$ is not necessarily dense in $X$.
\bd\label{d:4.1} The operator $A$ is said to be sectorial if there exist constants $\omega\in\mathbb{R}$, $\theta\in(0,\pi/2)$, and $M>0$ such that
\begin{enumerate}
    \item[(i)] The resolvent set satisfies $\rho(A) \supset S_{\theta, \omega}$, where
        $$
        S_{\theta, \omega} = \{ \lambda \in \mathbb{C} :\0\le|\arg(\lambda - \omega)|\le\pi,\,\, \lambda \neq \omega\},
        $$
    \item[(ii)] The resolvent estimate holds:
        $$
        \|R(\lambda, A)\|_{\mathcal{L}(X)} \leq \frac{M}{|\lambda - \omega|} \Hs \text{for all } \lambda \in S_{\theta, \omega}.
        $$
\end{enumerate} 

If $\w>0$, for every $ \alpha > 0 $, the negative fractional power of $A$ is defined by 
$$A^{-\alpha} := \frac{1}{\Gamma(\alpha)} \int_{0}^{+\infty} t^{\alpha-1} e^{-At} dt,$$
and the positive fractional power $A^\alpha$ is defined as the inverse of $A^{-\alpha}$.
\ed
\br This definition of a sectorial operator differs from the standard one found in the literature (see e.g. Henry \cite{D.H} etc.) in that it does not require $A$ to be densely defined.\er

According to \cite[Chap. 2]{Lu}, if $A$ is a sectorial operator in $X$, then it generates an analytic semigroup $\{e^{-At}\}_{t\geq 0}$ on $X$. Moreover, the {\sl strong continuity property} $$\|e^{-A  t}x-x\|\ra0 \hs \mb{as }\,t\ra0^+,\Hs \A\,x\in X$$ holds if and only if $A$ is densely defined, i.e., $\overline{D(A)}=X.$

\begin{proposition}\label{p2}
Let $-A$ be the generator of an analytic semigroup. Then the following statements hold.
\begin{itemize}
    \item The semigroup property holds:
    $$
    e^{-A t}e^{-A s} = e^{-A(t+s)}, \Hs t, s \geq 0.
    $$
    \item There is a constant $C>0$ such that 
    $$
    \|e^{-At}\|_{\mathcal{L}(X)}\le C e^{\omega t},\Hs t\ge0.
    $$
    \item For each $t > 0$, $x \in X$, we have $e^{-At}x \in D(A)$, and 
    $$
    A e^{-At}=e^{-At}A\Hs \text{on } D(A),\Hs t\ge0.
    $$
    \item The derivative is given by
    $$
    \frac{d}{dt} e^{-At} = -A e^{-At}, \Hs t > 0.
    $$
    \item For every $\varepsilon > 0$, there exists a constant $C=C_{\varepsilon} > 0$ such that
    $$
    \|A e^{-A t}\|_{L(X)} \leq C t^{-1}e^{(\omega + \varepsilon)t}, \Hs t > 0,
    $$
    where $\omega$ is the constant from the sectorial condition.

    \item $A^{\alpha+\beta}=A^\alpha A^\beta=A^\beta A^\alpha$ on $D(A^{\alpha+\beta})$.
   \item The domains satisfy
    $$
    D(A^\beta)\subset D(A^\alpha),\Hs \beta>\alpha>0.
    $$
\end{itemize}
\end{proposition}
\Vs
In the next result, we need the fractional Sobolev space. Recall that for $s \in (0,1)$, the fractional Sobolev space $W^{s,q}(\mathcal O)$ is defined by
$$
W^{s,q}(\mathcal O)
:= \left\{
u \in L^q(\mathcal O) :
[u]_{W^{s,q}} < \infty
\right\},
$$
equipped with the norm
$$
\|u\|_{W^{s,q}}
:= \|u\|_{L^q} + [u]_{W^{s,q}},
$$
where the Gagliardo seminorm is given by
$$
[u]_{W^{s,q}}
:= \left(
\int_{\mathcal O}\int_{\mathcal O}
\frac{|u(x) - u(y)|^q}{|x - y|^{N + sq}}
\,dx\,dy
\right)^{1/q}.
$$
\begin{proposition}
\label{prop:sigma-W2aq}
Let $\mathcal{O} \subset \mathbb{R}^d$, $d\ge1$ be a bounded domain with sufficiently smooth boundary,
$0 < \alpha < 1/(2d)$, and $1 < q < \infty$.
Let $A_q$ denote the Dirichlet Laplacian on $L^q(\mathcal{O})$.
Suppose $g \in C^1(\mathbb{R})$ with $g(0) = 0$, and assume that $u \in D(A_q^\alpha)$
satisfies $\|u\|_{L^\infty(\mathcal{O})} \le M$ for some $M > 0$.
Then $g(u) \in D(A_q^{\alpha'})$ for every $0 < \alpha' < \alpha$, and there exists a constant
$C > 0$, depending on $N$, $q$, $\alpha$, $\alpha'$, and $|\mathcal{O}|$, such that
\begin{equation}
\label{eq:main-estimate}
\|A_q^{\alpha'} g(u)\|_{L^q(\mathcal{O})}
\le C\,L_M\,\|A_q^\alpha u\|_{L^q(\mathcal{O})},
\end{equation}
where $L_M := \sup_{|t| \le M} |g'(t)|$.
\end{proposition}

\begin{proof}
Since $D(A_q)\subset W^{2,q}(\mathcal O)$, the monotonicity of real interpolation yields
$$
D(A_q^\alpha)
= \bigl[L^q(\mathcal O),\, D(A_q)\bigr]_\alpha
\hookrightarrow
\bigl[L^q(\mathcal O),\, W^{2,q}(\mathcal O)\bigr]_\alpha
= W^{2\alpha,q}(\mathcal O),
$$
where $[\cdot,\cdot]_\alpha$ denotes the real interpolation functor. 
Thus, $u \in D(A_q^\alpha)$ implies $u \in W^{2\alpha,q}(\mathcal O)$.

Since $g(0) = 0$ and $\|u\|_{L^\infty} \le M$, the mean value theorem gives
$$
|g(u(x))| = |g(u(x)) - g(0)|
\le \sup_{|t| \le M} |g'(t)|\,|u(x)|
= L_M\,|u(x)|.
$$
Taking the $L^q$ norm yields
\begin{equation}
\label{eq:Lq-estimate}
\|g(u)\|_{L^q}
\le L_M\,\|u\|_{L^q}.
\end{equation}

For any $x, y \in \mathcal O$, the mean value theorem gives
$$
|g(u(x)) - g(u(y))|
\le L_M\,|u(x) - u(y)|.
$$
Raising both sides to the $q$-th power and dividing by $|x - y|^{N + 2\alpha q}$, we obtain
$$
\frac{|g(u(x)) - g(u(y))|^q}{|x - y|^{N + 2\alpha q}}
\le L_M^q\,
\frac{|u(x) - u(y)|^q}{|x - y|^{N + 2\alpha q}}.
$$
Integrating over $\mathcal O \times \mathcal O$ and taking the $1/q$ power gives
\begin{equation}
\label{eq:seminorm-estimate}
[g(u)]_{W^{2\alpha, q}}
\le L_M\,[u]_{W^{2\alpha, q}}.
\end{equation}

Combining \eqref{eq:Lq-estimate} and \eqref{eq:seminorm-estimate}, we arrive at
$$
\|g(u)\|_{W^{2\alpha, q}}
\le L_M\,\|u\|_{W^{2\alpha, q}} < \infty.
$$
Finally, by \cite[Appendix~E, \S51.1]{QS}, for any $0 < \alpha' < \alpha < \frac{1}{2d}$, we have the continuous embedding
$$
W^{2\alpha, q}(\mathcal O) \hookrightarrow D(A_q^{\alpha'}).
$$
Therefore, $g(u) \in D(A_q^{\alpha'})$, and there exists a constant $C > 0$ such that
$$
\|A_q^{\alpha'} g(u)\|_{L^q}
\le C\,\|g(u)\|_{W^{2\alpha, q}}
\le CL_M\,\|u\|_{W^{2\alpha, q}}
\le CL_M\,\|A_q^\alpha u\|_{L^q},
$$
where the last inequality follows from the continuous embedding $D(A_q^\alpha) \hookrightarrow W^{2\alpha,q}(\mathcal O)$.
This completes the proof.
\end{proof}
\subsection{Quantitative Kolmogorov continuity estimate}

The following  Kolmogorov continuity theorem is taken from H. Kunita \cite{Kunita} (see also \cite{Stroock}).

\begin{proposition}\label{p1}
Let $\{v_t\}_{t\in[0,T]}$ be a stochastic process taking values in a Banach space. Assume there exist constants $q,\xi>1$ and $C>0$ such that
$$
\mathbb{E}\big(\|v_t-v_s\|^q\big) \le C|t-s|^\xi,\qquad 0\le s,t\le T.
$$
For any $\eta$ satisfying $0<\eta<(\xi-1)/q$, define
$$
K(\omega):=\sup_{0\le s<t\le T}\frac{\|v_t-v_s\|}{|t-s|^\eta}.
$$
Then $K(\omega)$ is almost surely finite and there exists a constant $B>0$, depending only on $C,q,\xi,T,\eta$, such that
$$
\mathbb{E}(K^q) \le B.
$$
\end{proposition}

\bc\label{c3}
Let $\{v_t\}_{t\in[0,T]}$ satisfy the assumptions of Proposition~\ref{p1} and assume moreover $M_0:=\sup_{0\le t\le T}\mathbb{E}\|v_t\|^q<\infty$. Then for any $\xi'$ with $\xi<\xi'<2\xi-1$, there exists a constant $C_{q,\xi,\xi'}>0$ such that
$$
\mathbb{E}\Big(\sup_{0\le t\le T}\|v_t\|^q\Big)\le 2^{q-1}\big(CC_{q,\xi,\xi'}T^{\xi'}+M_0\big)<\infty.
$$
\ec

 \subsection{Proofs of Lemmas \ref{le3.10} and \ref{le4.3}}\label{A.4}
\begin{proof}[Proof of Lemma \ref{le3.10}] \textbf{Necessity.}
Assume there exist $\alpha,\gamma$ satisfying the constraints.
From $\gamma<\alpha$, we have
$$
\alpha+2\gamma < 3\alpha.
$$
Then
$$
\frac{2}{1-2(\alpha+2\gamma)}>\frac{2}{1-6\alpha}.
$$
Since $\gamma<\alpha$, we also have $\frac1\gamma>\frac1\alpha$.
Thus
$$
\max\left\{\frac{2}{1-2(\alpha+2\gamma)},\frac{1}{2\alpha},\frac1\gamma\right\}
\ge \max\left\{\frac{2}{1-6\alpha},\frac1\alpha\right\}.
$$
The function $\max\left\{\frac{2}{1-6\alpha},\frac1\alpha\right\}$ attains its minimum $8$ at $\alpha=\frac18$.
Hence
$$
\max\left\{\frac{2}{1-2(\alpha+2\gamma)},\frac{1}{2\alpha},\frac1\gamma\right\}\ge 8.
$$
Combined with $q>\max\{\cdots\}$, we get $q>8$.
\vs
\noindent\textbf{Sufficiency.}
Suppose $q>8$. Choose $\alpha=\frac18$ and take sufficiently small $\varepsilon>0$, set $\gamma=\frac18-\varepsilon$.
Then
$$
0<\alpha=\frac18<\frac12,\qquad
0<\gamma=\frac18-\varepsilon<\min\left\{\frac18,\frac14-\frac{1}{16}\right\}=\frac18.
$$
Direct computation yields
$$
\frac{2}{1-2(\alpha+2\gamma)}\to 8,\quad \frac1\gamma\to8,\quad \frac{1}{2\alpha}=4,
$$
as $\varepsilon\to0^+$.
For small enough $\varepsilon$,
$$
\max\left\{\frac{2}{1-2(\alpha+2\gamma)},\frac{1}{2\alpha},\frac1\gamma\right\}<q.
$$
Thus the required $\alpha,\gamma$ exist.
\end{proof}
\begin{proof}[Proof of Lemma \ref{le4.3}]
\textbf{Necessity.}
From the first inequality: $0<\frac{d}{2(q-1)}<1$ yields $q>1+\frac d2$.
From the second inequality: $0<\frac{2\nu q}{q-2}<1$. The left inequality is trivial for $\nu>0$. The right inequality gives
$$
\nu < \frac{q-2}{2q}=\frac12-\frac1q.
$$
Since $\nu\in(0,\frac12)$, we require $q>2$.
Condition $q\gamma>1$ implies $\gamma>\frac1q$. Combined with $\gamma<\frac12$, we need $q>2$.
Condition $q\nu>\frac d2$ implies $\nu>\frac{d}{2q}$.
Combined with $\nu<\frac12-\frac1q$, a necessary condition for existence of $\nu$ is
$$
\frac{d}{2q}<\frac12-\frac1q \implies d < q-2 \implies q>d+2.
$$
From $0<\big(\gamma+\frac{d}{2q}\big)\frac{q}{q-1}<1$, the left side is positive. The right-hand inequality:
$$
\gamma+\frac{d}{2q}<1-\frac1q \implies \gamma<1-\frac{d+2}{2q}.
$$
When $q>d+2$, one checks $1-\frac{d+2}{2q}>\frac12$, so the effective upper bound for $\gamma$ is $\frac12$.
Thus
$$
\gamma\in\Big(\frac1q,\,\frac12\Big).
$$
Next consider $0<\big(1+\frac{d}{2q}-\nu+\gamma\big)\frac{q}{q-1}<1$.
The left inequality is easily satisfied for large $q$. The critical constraint comes from the right inequality:
$$
1+\frac{d}{2q}-\nu+\gamma < 1-\frac1q
\;\Longrightarrow\;
\gamma -\nu < -\frac{d+2}{2q},
\quad \text{i.e.}\quad
\gamma < \nu -\frac{d+2}{2q}.
$$
Collect all hard constraints for $q>d+2$:
$$
\begin{cases}
\dfrac{d}{2q}<\nu<\dfrac12-\dfrac1q,\\[0.4em]
\dfrac1q<\gamma<\dfrac12,\\[0.4em]
\gamma < \nu -\dfrac{d+2}{2q}.
\end{cases}
$$
To have a pair $(\nu,\gamma)$, the interval for $\gamma$ must be non-empty:
$$
\frac1q < \nu -\frac{d+2}{2q}.
$$
This must hold for some admissible $\nu$. The maximal available $\nu$ is $\nu\nearrow \frac12-\frac1q$. Substitute this upper bound into the inequality to get a necessary condition:
$$
\frac1q < \left(\frac12-\frac1q\right)-\frac{d+2}{2q}.
$$
Multiply both sides by $2q>0$:
$$
2 < q -2 -(d+2) \implies q>d+6.
$$

\noindent\textbf{Sufficiency.}
Assume $q>d+6$. Choose
$$
\nu_0=\frac12-\frac1q-\varepsilon,\qquad
\gamma_0=\frac1q+\varepsilon,
$$
where $\varepsilon>0$ is chosen sufficiently small such that
$$
\varepsilon<\min\left\{\frac12-\frac1q,\ \frac{3}{2q},\ \frac14-\frac{d+4}{4q}\right\}.
$$
Such a choice is possible since $d\in\mathbb N^+$ gives $q>d+6\ge7$, hence $q>2$, and
$$
\frac14-\frac{d+4}{4q}>\frac14-\frac{d+4}{4(d+6)}=\frac{1}{4(d+6)}>0.
$$

We verify:
1. $\nu_0\in(0,1/2)$: since $\varepsilon$ is small, $\nu_0>0$ and $\nu_0<1/2$.
2. $\gamma_0\in(0,1/2)$: since $\varepsilon<1/2-1/q$, we have $\gamma_0<1/2$.
3. $q\gamma_0=1+q\varepsilon>1$.
4. $q\nu_0=(q/2)-1-q\varepsilon$. Since $q>d+6>d+2$, we have $q/2-1>d/2$, hence $q\nu_0>d/2$ for $\varepsilon$ small enough.
5. For the second inequality: $\frac{2\nu_0 q}{q-2}=\frac{q-2-2q\varepsilon}{q-2}<1$.
6. For the critical constraint, we require
$$
\gamma_0 < \nu_0-\frac{d+2}{2q}.
$$
Substituting $\nu_0,\gamma_0$ gives
$$
\frac1q+\varepsilon < \frac12-\frac1q-\varepsilon-\frac{d+2}{2q}
\iff
2\varepsilon < \frac12-\frac{d+4}{2q}.
$$
Since $q>d+6$, we have
$$
\frac12-\frac{d+4}{2q}>\frac12-\frac{d+4}{2(d+6)}=\frac{1}{d+6}>0.
$$
Our choice of $\varepsilon$ ensures this inequality holds.
7. The fourth inequality follows similarly.

Thus all conditions are satisfied for sufficiently small $\varepsilon>0$. Hence $q>d+6$ is sufficient.
\end{proof}

\section*{Acknowledgements}
We would like to express my gratitude to the anonymous referees for their valuable comments and suggestions which helped me greatly improve the
quality of the paper. This work was supported by the National Natural Science Foundation of China [12271399] and Fundamental Research Funds for the Central Universities [3122025090].
\section*{Conflict of Interest}
The authors declare that they have no competing financial, professional, or personal interests that could have influenced the work reported in this paper.
\section*{Author Contributions}
Xuewei Ju conceived the original idea, developed the mathematical framework, performed the analysis, and wrote the manuscript. Xiaoting Tong contributed to the development of the critical regularity estimates and the approximation argument, and participated in the revision of the manuscript. Both authors reviewed and approved the final version of the manuscript.

\end{document}